\documentclass[10pt]{amsart}
\usepackage{amsmath,amsfonts,amsthm}
\usepackage{color}
\usepackage[]{graphics}
\usepackage{verbatim}
\usepackage{eepic,epic}
\usepackage{epsfig,subfigure,epstopdf}
\usepackage[colorlinks,linktocpage,linkcolor=blue]{hyperref}
\usepackage{slashbox}
\usepackage{colortbl}

\newtheorem{theorem}{Theorem}[section]
\newtheorem{lemma}{Lemma}[section]

\newtheorem{remark}{Remark}[section]

\numberwithin{equation}{section}

\newcommand{\al}{\alpha}
\newcommand{\fy}{\varphi}
\newcommand{\fyy}{\psi}
\newcommand{\hfy}{\widehat\varphi}

\newcommand{\la}{\lambda}
\newcommand{\ep}{\epsilon}
\newcommand{\rb}[1]{\raisebox{1.5ex}[0pt]{#1}}

\def\dH#1{\dot H^{#1}(\Omega)}
\def\Dmu{\mathrm{D}_t^{[\mu]}}
\def\Dal{{D_t^\al}}

\def\Om{\Omega}
\def\II{(\Om)}
\def\cS{\mathcal S}
\def\hphi{\widehat \phi}
\def\hg{\widehat g}

\def\G{\Gamma}

\def\hy{\widehat y}

\begin{document}
\title[Error Estimates for Distributed order FPDE]
{Error Estimates for Approximations of Distributed Order Time Fractional Diffusion with Nonsmooth Data }
\author {Bangti Jin \and Raytcho Lazarov \and Dongwoo Sheen  \and Zhi Zhou}
\address{Department of Computer Science, University College London, Gower Street,
London WC1E 6BT, UK ({bangti.jin@gmail.com})}
\address{Department of Mathematics, Texas A\&M University, College Station,
TX 77843-3368, USA \\
(lazarov@math.tamu.edu, zzhou@math.tamu.edu)}
\address{Department of Mathematics, Seoul National University,
Seoul 151-747, Korea (dongwoosheen@gmail.com)}
\date{started September 2014; today is \today}

\begin{abstract}
In this work, we consider the numerical solution of an initial boundary value problem for the distributed
order time fractional diffusion equation. The model arises in the mathematical modeling of ultra-slow
diffusion processes observed in some physical problems, whose solution decays only logarithmically
as the time $t$ tends to infinity. We develop a space semidiscrete scheme based on the standard
Galerkin finite element method, and establish error estimates  optimal with respect to data
regularity in $L^2\II$ and $H^1\II$ norms for both smooth and nonsmooth initial data. Further, we propose
two fully discrete schemes, based on the Laplace transform and convolution quadrature generated by the
backward Euler method, respectively, and provide optimal convergence rates in the $L^2\II$ norm,
which exhibits exponential convergence and first-order convergence in time, respectively.
Extensive numerical experiments are provided to verify the error estimates for both smooth and
nonsmooth initial data, and to examine the asymptotic behavior of the solution. 
\end{abstract}

% 65M60 (35R11 65M15)
\maketitle

%%%%%%%%%%%%%%%%%%%%%%%%%%%%%%%%%%%%%%%%%%%%%%%%%%
\section{Introduction}\label{sec:intro}
%%%%%%%%%%%%%%%%%%%%%%%%%%%%%%%%%%%%%%%%%%%%%%%%%%
We consider an initial-boundary value problem for the following distributed order time fractional diffusion equation
for $u(x,t)$:
\begin{alignat}{3}\label{eqn:fde}
   \Dmu u -\Delta u&= f&&\quad \text{in  } \Om&&\quad T \ge t > 0,\notag\\
   u&=0&&\quad\text{on}\  \partial\Om&&\quad T \ge t > 0,\\
    u(0)&=v&&\quad\text{in  }\Om,&&\notag
\end{alignat}
where $\Om$ is a bounded convex polygonal domain in $\mathbb R^d\,(d=1,2,3)$ with a boundary $\partial\Om$, $v$ is a given
function on $\Om$, and $T>0$ is a fixed value. Here,
$\Dmu u$ denotes the distributed order fractional derivative of $u$ in time $t$ (with respect to the
weight function $\mu$) defined by
\begin{equation}\label{DF}
   \Dmu u(t)= \int_0^1 \Dal u(t) \mu(\al)\,d\al,
\end{equation}
where $\Dal u$, $0<\al<1$, denotes the left-sided Caputo fractional derivative of order $\al$ with respect to $t$ and
it is defined by (see, {\it e.g.} \cite[p. 91]{KilbasSrivastavaTrujillo:2006})
\begin{equation}\label{McT}
   \Dal u(t)= \frac{1}{\Gamma(1-\al)} \int_0^t(t-s)^{-\al}\frac{d}{ds}u(s)\, ds,
\end{equation}
where $\Gamma(\cdot)$ denotes Euler's Gamma function defined by $\Gamma(x)=\int_0^\infty t^{x-1}e^{-t}dt$, for all $x>0$.
In this paper we consider the case that $\mu \in C[0,1]$ is a nonzero nonnegative weight function with $0\leq  \mu<1$, $\mu(0)\mu(1)>0$.

In the last three decades, fractional calculus has been extensively studied and successfully employed to model anomalous diffusion,
in which the mean squared variance grows faster (superdiffusion) or slower (subdiffusion) than that in a Gaussian process. The
subdiffusion model, which is a diffusion equation involving a Caputo fractional derivative $D_t^{\alpha_0}u $ of order $ \al_0\in (0,1)$ in time:
\begin{equation}\label{eqn:subdiff}
  D_t^{\alpha_0} u -\Delta u= f \quad \text{in  } \Om, \quad T \ge t > 0,
\end{equation}
is often employed to model subdiffusion processes in which the mean squared variance grows at a sublinear (power type)
rate, slower than the linear growth in a Gaussian process for normal diffusion. Formally, the subdiffusion model \eqref{eqn:subdiff}
can be recovered from the distributed order model \eqref{eqn:fde} with a singular weight $\mu(\alpha)=\delta(\alpha-\alpha_0)$, where
$\delta(\alpha-\alpha_0)$ is a Dirac delta function at $\alpha_0$. Physically, the subdiffusion process can be characterized by a unique diffusion
exponent (commonly known as Hurst exponent) showing the time dependence of the characteristic displacement \cite{ChechkinGorenfloSokolov:2002}.
In practice, the physical process may not possess a unique Hurst exponent, and the distributed order model \eqref{eqn:fde} provides a
flexible framework for describing a host of continuous and nonstationary signals \cite{ChechkinGorenfloSokolov:2002,ChechkinGorenfloSokolovGonchar:2003,
SokolovChechkinKlafter:2004}. Problem \eqref{eqn:fde} is frequently applied to describe ultraslow diffusion,
where the mean squared variance grows only logarithmically with time, {\it e.g.}, Sinai model \cite{Sinai:1982}.
The distributed-order fractional model arises often in disordered media,
and has been successfully used in several applications. For example, Caputo \cite{Caputo:2001}
proposed the use of the distributed order derivative in generalizing the
stress-strain relation in dielectrics, and Atanakovic {\it et al.} \cite{AtanakovicPilipovicZorica:2011}
suggested a distributed order wave equation as the constitutive relation for
viscoelastic materials to describe stress relaxation in a rod.

In recent years, the theoretical study of problem \eqref{eqn:fde} has attracted some
attention \cite{MeerschaertScheffler:2006,Kochubei:2008,mainardi2008time,
MeerschaertNaneVellaisamy:2011,Luchko:2009,GorenfloLuchkoStajanovic:2013,LiLuchkoYamamoto:2014,JiaPengLi:2014,Bazhlekova_2015arXiv}.
Kochubei \cite{Kochubei:2008} made some early contributions to the rigorous analysis of the model \eqref{eqn:fde}, by
constructing fundamental solutions to the problem and establishing their positivity and subordination property.
Mainardi \textit{et al.} \cite{mainardi2008time} studied the existence of a solution, asymptotic behavior, and positivity \emph{etc.}
for the case $\mu(0) \ge 0$ and $\int_0^1 \mu(\alpha) d\alpha =c > 0$. Meerschaert and Scheffler
\cite{MeerschaertScheffler:2006} {gave a stochastic model for } ultraslow diffusion, based on random walks with a random
waiting time between jumps whose probability tail falls off at a logarithmic rate. Meerschaert {\it et al.}
\cite{MeerschaertNaneVellaisamy:2011} provided explicit strong solutions and their stochastic analogues. Luchko
\cite{Luchko:2009} showed a weak maximum principle for the problem. Li \emph{et al.} \cite{LiLuchkoYamamoto:2014}
established sharp asymptotic behavior of the solution for $t\to0$ and $t\to \infty$, in the case of continuous density $
\mu$ with $\mu(1)>0$. Jia \textit{et al.} \cite{JiaPengLi:2014} studied the well-posedness of a Cauchy problem for an abstract
distributed-order differential equation using a functional calculus approach. Very recently, Bazhlekova \cite{Bazhlekova_2015arXiv}
analyzed problem \eqref{eqn:fde} for $\mu \in C[0,1]$, $\mu \ge 0$ and $\mu(\alpha)\not = 0$ on a set of positive measure.

The solution to the model \eqref{eqn:fde} is rarely available in closed form, which necessitates the development
of efficient numerical schemes, to enable the successful use of the model \eqref{eqn:fde} in practice. Despite the
extensive studies on the simpler subdiffusion model \eqref{eqn:subdiff} (\textit{cf.}  \cite{LinXu:2007,
McleanMustapha:2014,ZhangSunLiao:2014,CuestaLubichPlencia:2006,ZengLiLiuTurner:2013,MustaphaAbdallahFurati:2014,JinLazarovZhou:2014a,JinLazarovZhou:2014L1}
for an incomplete list of works on the numerical approximation of the Caputo fractional derivative $\Dal u(t)$), there are only
very few studies \cite{DiethelmFord:2009,Katsikadelis:2014,MorgadoRebelo:2015} on the distributed order
model \eqref{eqn:fde}. Diethelm and Ford \cite{DiethelmFord:2009} developed a numerical scheme for distributed order
fractional ODEs. It approximates the distributed order derivative $\Dmu u(t)$ by quadrature, leading to a multi-term
time fractional ODE, which can then be solved by fractional multi-step methods. Error estimates of the approximation
were discussed in \cite{DiethelmFord:2009}. Such a technique was also employed to solve nonlinear distributed-order fractional
ODEs in \cite{Katsikadelis:2014}, but without any analysis. Just recently, Morgado and Rebelo \cite{MorgadoRebelo:2015}
developed an implicit finite difference method for the model \eqref{eqn:fde} with a Lipschitz nonlinear source term in
one space dimension. The scheme is based on a quadrature approximation of $\Dmu u(t)$ together with the backward finite
difference approximation for the Caputo derivative $\Dal u(t)$, and the second-order finite difference approximation in
space. The stability of the scheme, and a convergence rate $O(h^2+\tau+(\delta\alpha)^2)$ (with $h,\tau$ and $\delta\alpha$
being the mesh size, time step size and step size for quadrature rule, respectively) were established under the assumption
that the solution $u$ is $C^2$ in time and $C^4$ in space and the weight function $\mu(\alpha)$ is
sufficiently regular. In view of the limited smoothing property of the solution operator, \emph{cf.} Theorem \ref{thm:reg}
below, the regularity required by the convergence analysis is restrictive, especially for nonsmooth data. To the best of our
knowledge, the development of robust numerical schemes for the model \eqref{eqn:fde} with nonsmooth data and
their rigorous analysis have not been carried out, despite its immense practical importance, {\it e.g.}, in solving
inverse and/or optimal control problems \cite{JinRundell:2014}.

In this work, we develop a Galerkin finite element method (FEM) for problem \eqref{eqn:fde} and establish optimal
(with respect to data regularity) error estimates for both smooth and nonsmooth initial data $v$. The approximation is based on
the finite element space $X_h$ of continuous piecewise linear functions over a family of shape regular quasi-uniform
partitions $\{\mathcal{T}_h\}_{0<h<1}$ of the domain $\Om$ into $d$-simplexes, where $h$ is the maximum diameter of the partition.
Then the space semidiscrete Galerkin FEM for problem \eqref{eqn:fde} is given by: find $ u_h (t)\in X_h$ such that
\begin{equation}\label{eqn:semidiscrete}
  \begin{split}
   {(\Dmu u_{h},\chi)}+ a(u_h,\chi)&= {(f, \chi)},
                      \quad \forall \chi\in X_h,\ T \ge t >0,
   \quad u_h(0)=v_h,
  \end{split}
\end{equation}
where $(\cdot,\cdot)$ denotes the $L^2\II$-inner product,
$a(u,w)=(\nabla u, \nabla w) ~~ \text{for}\ u, \, w\in H_0^1(\Om)$, and $v_h \in X_h$ is an approximation of the
initial data $v$. Our default choices for $v_h$ are the $L^2(\Om)$-projection $v_h=P_hv$, for $v\in L^2(\Om)$, and the
Ritz projection $v_h=R_hv$, for $Av\in L^2(\Omega)$, where $A=-\Delta$ with a homogeneous Dirichlet boundary condition.
Further, we develop two fully discrete schemes based on the Laplace transform and convolution quadrature
%\cite{gavrilyuk2002mathcal,
%Lopez-FernandezPalencia:2006,mclean2010numerical,SheenSloanThomee:2000,SheenSloanThomee:2003,WeidemanTrefethen:2007}
%and convolution quadrature \cite{Lubich:1988, LubichSloanThomee:1996}
generated by the backward Euler method, and provide optimal error estimates for both space semidiscrete and fully discrete schemes.

Our main contributions are as follows. First, in Theorem \ref{thm:reg}, we establish the sharp regularity estimates
for the solution to problem \eqref{eqn:fde}. The proof relies essentially on various refined properties of the kernel function $w(z)$
defined in \eqref{eqn:w} in Lemmas \ref{lem:sector-w}-\ref{lem:key2}, which also enable one to apply the
established techniques for analyzing the semidiscrete and fully discrete schemes.
Second, in Theorems \ref{thm:nonsmooth-initial-op} and  \ref{thm:smooth-initial-op},
%respectively, for smooth initial data $Av\in L^2(\Om)$ and nonsmooth initial data $v\in L^2(\Om)$,
we derive the following error estimates for the space semidiscrete Galerkin scheme \eqref{eqn:semidiscrete} for $t\in(0,T]$:
\begin{equation*}
 \|u(t)-u_h(t)\|_{L^2(\Om)} + h\|\nabla (u(t)-u_h(t))\|_{L^2(\Om)} \leq
\begin{cases}
c_Th^{2} \left|t\log \frac{2T}{t}\right|^{-1} \|v\|_{L^2(\Om)}\quad & \text{if } v\in L^2(\Om),  \\
ch^{2} \|Av\|_{L^2(\Omega)}\quad &\text{if } Av\in L^2(\Omega).
\end{cases}
\end{equation*}
For initial data $v\in L^2(\Om)$, the estimates deteriorate as the time $t$ approaches $0$, with an extra {$\frac1{|\log t|}$} factor
in comparison with that for the standard diffusion case \cite{Thomee:2006}. Third, we develop a fully discrete
scheme based on the Laplace transform. It relies on a contour representation of the semidiscrete solution with a hyperbolic contour,
and trapezoidal quadrature, cf. Theorem \ref{thm:error_fully}. Specifically,
the fully discrete solution $U_{N,h}(t)$ with $N+1$ quadrature points satisfies the following error bound for $t\in(0,T]$:
\begin{equation*}
 \|u(t) - U_{N,h}(t) \|_{L^2(\Om)} \leq
\begin{cases}
c_T\left(e^{-c_1N } +  h^{2} \left|t\log \frac{2T}{t}\right|^{-1}\right) \|v\|_{L^2(\Om)}\quad & \text{if } v\in L^2(\Om),  \\
c\left(e^{-c_1N }+ h^{2}\right) \|Av\|_{L^2(\Om)}\quad &\text{if } Av\in {L^2(\Om)}.
\end{cases}
\end{equation*}
Last, we develop a second fully discrete scheme based on convolution quadrature, generated by the backward Euler method, 
and in Theorem \ref{thm:error_fully_CQ}, establish the first-order convergence of the scheme for both smooth and nonsmooth 
initial data. For example, for nonsmooth initial data $v\in L^2(\Om)$, the fully discrete solution $U_h^n$ approximating 
the continuous solution $u(t_n)$, $t_n\in(0,T]$ (on a uniform grid in time with a step size $\tau$) satisfies the following bound:
\begin{equation*}
   \| u(t_n)-U_h^n \|_{L^2(\Om)} \le c_T\left(\tau + h^2 \left|\log \tfrac{2T}{t}\right|^{-1}\right)t_n^{-1}  \| v\|_{L^2\II}.
\end{equation*}
It is worth noting that all error estimates are nearly optimal and expressed in terms of the regularity of the initial data directly,
and fully verified by extensive numerical experiments. Theoretically, these results extend our earlier studies
\cite{JinLazarovZhou:2013,JinLazarovZhou:2014a,JinLazarovLiuZhou:2014} on the subdiffusion model \eqref{eqn:subdiff},
which contribute to the development and rigorous analysis of robust numerical schemes for the distributed order model \eqref{eqn:fde}.

The model \eqref{eqn:fde} is closely related to parabolic equations with a
positive type memory term, for which there are many important studies on numerical schemes based on convolution quadrature \cite{Lubich:1988,LubichSloanThomee:1996,CuestaLubichPlencia:2006} and Laplace transform \cite{gavrilyuk2002mathcal,
Lopez-FernandezPalencia:2006,SheenSloanThomee:2000,SheenSloanThomee:2003,McLeanSloanThomee:2006,WeidemanTrefethen:2007,
mclean2010numerical}. For example, Cuesta \emph{et al.} \cite{CuestaLubichPlencia:2006} developed an abstract framework for
analyzing convolution quadrature generated by the backward Euler method and second-order backward difference. It
covers also inhomogeneous and nonlinear problems. Their analysis uses the generating function, and the Laplace transform 
involves $w(z) =z^\alpha,\ 0<\alpha <1$. McLean and Thom\'ee \cite{mclean2010numerical} studied the Laplace transform method 
for a fractional order model, whose Laplace transform involves $w(z)=z^\alpha,\ -1<\alpha <1$. These interest works have
inspired the current work on the model \eqref{eqn:fde}. However, the existing error analysis 
does not cover directly the model \eqref{eqn:fde}, due to the general kernel function
involved, cf. \eqref{eqn:w}. Instead, we shall opt for the general strategy outlined in \cite{LubichSloanThomee:1996}, by
deriving various refined estimates for the kernel function, especially identifying the suitable condition on the weight 
function $\mu$. These estimates are also essential for analyzing the Laplace transform approach.

The rest of the paper is organized as follows. In Section \ref{sec:prelim}, we recall the solution theory of the mathematical model
\eqref{eqn:fde} following \cite{Kochubei:2008,LiLuchkoYamamoto:2014}. In Section \ref{sec:semi}, we develop a space
semidiscrete Galerkin  scheme, and provide optimal error estimates. Two fully discrete schemes, based on the
Laplace transform and convolution quadrature, are given in Sections \ref{sec:Laplace} and \ref{sec:Conv-quad}, respectively.
Finally, to test and to verify the convergence theory, we present in Section \ref{sec:numer} extensive numerical experiments.
Throughout, the notation $c$, with or without a subscript, denotes a generic constant, which may differ
at different occurrences, but it is always independent of the mesh size $h$, the number $N$ of quadrature points,
and time step size $\tau$.

\section{Solution theory}\label{sec:prelim}
In this part, we discuss the solution theory of problem \eqref{eqn:fde} using the Laplace transform. The stability
estimates will play an essential role in developing optimal error estimates. We denote by $\ \widehat{}\ $ the
Laplace transform. First we recall the following well known relation
\cite[Lemma 2.24, p. 98]{KilbasSrivastavaTrujillo:2006}
\begin{equation}\label{eqn:Laplacetransf}
  \widehat {\partial_t^\al u}= z^\al \widehat{u}-z^{\al-1}u(0).
\end{equation}
Next we denote by $A$ the operator $-\Delta$ with a homogeneous Dirichlet boundary condition with
a domain $D(A) = H_0^1(\Om)\cap H^2(\Om)$. The $H^2(\Om)$ regularity of the elliptic problem is
essential for the error analysis below and it follows from the convexity assumption on the domain $\Om$. It is well known that
the operator $A$ generates a bounded analytic semigroup of angle $\pi/2$, {\it i.e.}, for any
$\theta\in(\pi/2,\pi)$ \cite[p.\,321, Proposition C.4.2]{Hasse_book}
\begin{equation}\label{eqn:resol}
 \|(zI + A)^{-1}\|\leq \frac{1}{|\Im(z)|}\le \frac1{|z\sin(\theta)|} \quad \forall z\in \Sigma'_{\theta},
\end{equation}
where $\Sigma'_{\theta}$ is a sector with the origin excluded, \textit{i.e.},
\begin{equation*}
  \Sigma_{\theta} = \{z\in\mathbb{C}: |\mathrm{arg}(z)| < \theta\},\quad
  \Sigma'_{\theta} = \Sigma_{\theta} \setminus \{0\}.
\end{equation*}
Now it follows from \eqref{eqn:fde} and \eqref{eqn:Laplacetransf} that
\begin{equation}\label{eqn:laptrans}
    z w(z) \widehat u(z) + A\widehat u(z)=w(z) v,
\end{equation}
where the function $w(z)$ is defined by
\begin{equation}\label{eqn:w}
  w(z) = \int_0^1 z^{\al-1} \mu(\al) \, d\al.
\end{equation}
By means of the inverse Laplace transform, the solution $u(t)$ can be represented by
\begin{equation}\label{eqn:interep}
 u(t)= S(t) v:=\frac{1}{2\pi \mathrm{i}} \int_{\Gamma_{\theta,\delta}} e^{zt}H(z) v \,dz,
\end{equation}
where the kernel $H(z)$ is defined by
\begin{equation*}
 H(z) =  (zw(z)I+A)^{-1}w(z),
\end{equation*}
and the contour $\Gamma_{\theta,\delta}$ by
\begin{equation}\label{eqn:contour}
  \Gamma_{\theta,\delta}=\left\{z\in \mathbb{C}: |z|=\delta, |\arg z|\le \theta\right\}\cup
  \{z\in \mathbb{C}: z=\rho e^{\pm i\theta}, \rho\ge \delta\}.
\end{equation}

%\begin{remark}
%Cuesta et al \cite{CuestaLubichPlencia:2006} and Lubich et al. \cite{mclean2010numerical} considered an evolution
%equation with a positive-type memory term and proposed quadrature schemes based on the first and second-order
%backward difference methods. Their analyses use Laplace transform which imply that their problem can be regarded
%as a special case such that $w(z) =z^\alpha, 0<\alpha <1,$ but with inclusion of a source term.
%McLean and Thom\'ee \cite{mclean2010numerical} studied the Laplace transform method for a fractional order
%evolution equation whose Laplace transform is a special case such that
%$w(z) =z^\alpha, -1<\alpha <1,$ but again including a source term. These works focused on
%the inhomogeneous problem, applying three different approaches, whose
%analyses are more or less complete, but much more complicated than those for
%homogeneous evolution equations.
%%For the subdiffusion case for $w(z) =z^\alpha, 0<\alpha <1$,
%%recently, Mustapha \cite{McleanMustapha:2014,MustaphaAbdallahFurati:2014} analyzed modified backward Euler, or
%%Crank--Nicolson convolution quadrature schemes.
%Our concern in the present paper is to
%consider only homogeneous evolution equations with a more general form of
%$w(z)$.
%%Mclean and Mustapha \cite{McleanMustapha:2014} considered the homogeneous
%%subdiffusion model for $w(z) =z^\alpha, 0<\alpha <1$ with non-smooth initial data.
%\end{remark}

We begin by discussing the regularity estimates of the solution. To this end, we first give a few elementary properties
of the function $w(z)$. The first is the sector-preserving property, which enables applying the
resolvent estimate \eqref{eqn:resol} in the error analysis to be developed in Sections \ref{sec:semi}-\ref{sec:Conv-quad}.
\begin{lemma}\label{lem:sector-w}
Let $\theta\in (\pi/2,\pi)$ and assume that $\mu(0)\mu(1)>0$. Then $zw(z) \in \Sigma_{\theta'}$
with $\theta' \in (\pi/2,\pi)$ for all $z \in \Sigma_{\theta}$ and $\theta'$ depends only on $\mu$ and $\theta$.
\end{lemma}
\begin{proof}
Let $z=re^{\mathrm{i}\fy}$ with $ \fy\in [-\theta,\theta]$. If $\fy \in (-\pi/2,\pi/2)$, we have
\begin{equation*}
 \mathrm{\Re}(zw(z)) = \int_0^1 r^{\al} \cos(\al\fy) \mu(\al) \, d\al > 0.
\end{equation*}
It suffices to consider the case $\fy \in [\pi/2,\theta]$. First we claim that there exists an
$r_0\in(0,1)$ only dependent on $\mu$ such that  $\mathrm{\Re}(zw(z))  > 0$ for all $r<r_0$.
By the assumption $\mu(0)>0$, we can find a small $\ep_0>0$ such that
$\min_{\al\in[0,\ep_0]}\cos(\al\pi)\mu(\al)=\delta_0>0$. Hence
\begin{equation}\label{eqn:ep0}
\begin{split}
 \mathrm{\Re}(zw(z)) %&= \int_0^1 r^{\al} \cos(\al\fy) \mu(\al) \, d\al \\
 &\ge \int_0^{\ep_0} r^{\al} \cos(\al\fy) \mu(\al) \, d\al - \int_{\ep_0}^{1} r^{\al}|\cos(\al\fy)|\mu(\al) \, d\al\\
 %&\ge \int_0^{\ep_0} r^{\al} \cos(\al\theta) \mu(\al) \, d\al - \| \mu \|_{C[0,1]}\int_{\ep_0}^{1} r^{\al}\, d\al\\
 &\ge \delta_0\int_0^{\ep_0} r^{\al} \, d\al - \| \mu \|_{C[0,1]}\int_{\ep_0}^{1} r^{\al}\, d\al\\
 &\ge -(\ln r)^{-1}\left[ \delta_0-r^{\ep_0}(\delta_0+\|\mu \|_{C[0,1]})   \right].
\end{split}
\end{equation}
Then direct calculation yields
\begin{equation*}
\begin{split}
   \delta_0-r^{\ep_0}(\delta_0+\|\mu \|_{C[0,1]}) >0 \quad \forall r<r_0=:\left(\frac{\delta_0}{\delta_0+\|\mu \|_{C[0,1]}}\right)^{1/\ep_0}\in(0,1),
\end{split}
\end{equation*}
and the desired claim $\mathrm{\Re}(zw(z))  > 0$ follows. Now we consider the case $r\ge r_0$ and  $\fy \in [\pi/2,\theta]$.
In fact,
\begin{equation*}
 \begin{split}
  | \tan(\arg(zw(z))) | &= \frac{|\int_0^1 r^{\al} \sin(\al\fy) \mu(\al) \, d\al|}{|\int_0^1 r^{\al} \cos(\al\fy) \mu(\al) \, d\al|}
  \ge \frac{\int_0^1 r^{\al} \sin(\al\fy) \mu(\al) \, d\al}{\| \mu  \|_{C[0,1]}\int_0^1 r^{\al} \, d\al}.
 \end{split}
\end{equation*}
In view of the assumption $\mu(1) > 0$ and $\mu \in C[0,1]$, we may find a small $\ep_1>0$
such that $\min_{\al\in[1-\ep_1,1]} \mu(\al)\ge \delta_1>0$ and
\begin{equation*}
 \begin{split}
  \int_0^1 r^{\al} \sin(\al\fy) \mu(\al) \, d\al  &\ge \int_{1-\ep_1}^1 r^{\al} \sin(\al\fy) \mu(\al) \, d\al \\
  &\ge  \int_{1-\ep_1}^1 r^{\al} \sin(\theta) \mu(\al) \, d\al
  %\sin(\fy) \int_{1-\ep_1}^1 r^{\al} \mu(\al) \, d\al \\
  \ge  \delta_1 \sin(\theta) \int_{1-\ep_1}^1 r^{\al} \, d\al.% \ge  \sin(\theta)\delta_1  \int_{1-\ep_1}^1 r^{\al}\,d\al.
 \end{split}
\end{equation*}
For $r\ge r_0$, clearly there holds
\begin{equation}\label{eqn:rger0}
 \begin{split}
\int_{1-\epsilon_1}^1 r^{\al}\,d\al &= \int_{1-\epsilon_1}^1 r_0^\al\left(\frac{r}{r_0}\right)^{\al}\,d\al
\ge r_0 \int_{1-\epsilon_1}^1 \left(\frac{r}{r_0}\right)^{\al}\,d\al\\
& \ge r_0\ep_1 \int_{0}^1 \left(\frac{r}{r_0}\right)^{\al}\,d\al
\ge r_0\ep_1 \int_{0}^1 r^{\al}\,d\al.
 \end{split}
\end{equation}
Then we have
\begin{equation}\label{eqn:arg}
  | \tan(\arg(zw(z))) | \ge   \delta_1\sin(\theta) r_0\ep_1 /\| \mu  \|_{C[0,1]} =:c'.
\end{equation}
Hence $zw(z) \in \Sigma_{\theta'} $ with $\theta'=\pi-\arctan(c')$.
\end{proof}

\begin{remark}\label{rem:eparg}
We note that the constants $\delta_1$, $r_0$  and $\ep_1$ in \eqref{eqn:arg} are all independent of the choice
$\theta$. Hence in case of $\theta=\pi-\epsilon$ for a small $\ep>0$,  $zw(z) \in \Sigma_{\theta'} $ with
$\theta'=\pi-\arctan(c \sin(\theta))= \pi-\arctan(c \sin(\ep)) \approx\pi-c \epsilon, $
where $c=\delta_1 r_0\ep_1 /\| \mu  \|_{C[0,1]}$.
\end{remark}

The second result  is an upper bound on the kernel $w(z)$, which can be obtained by an elementary calculation.
\begin{lemma}\label{lem:mu}
 Let $\mu\in C[0,1]$ be a nonnegative function. Then there holds
 $$ |w(z)| \le \| \mu  \|_{C[0,1]}\frac{|z|-1}{|z|\log|z|}.  $$
\end{lemma}

The third result gives a lower bound on the function $zw(z)$.
\begin{lemma}\label{lem:key2}
Let $\theta\in(\pi/2,\pi)$ and assume that $\mu(0)\mu(1)>0$. Then there exists a constant $c>0$ dependent
only on $\theta$ and $\mu$  such that for any $z\in \Sigma_{\theta}'$
\begin{eqnarray}
  &&|zw(z)|  \ge c\int_0^1 r^\al\,d\al= c\frac{|z|-1}{\log|z|} \label{eqn:key2}, \\
  &&|z|w(|z|)\geq |zw(z)| \ge c|z|w(|z|).\label{eqn:key3}
\end{eqnarray}
\end{lemma}
\begin{proof}
Let $z=re^{\mathrm{i}\fy}$. Using $\mu(1) > 0$ and $\mu \in C[0,1]$, we can find a small $\ep_1>0$ such that
$\min_{\al\in[1-\ep_1,1]} \mu(\al) \ge \delta_1>0$. Then we have for all $r\ge1$
\begin{equation*}
 \begin{split}
\int_{0}^1 r^{\al}\mu(\al)\,d\al &\ge \int_{1-\epsilon_1}^1r^{\al}\mu(\al)\,d\al
\ge \delta_1\int_{1-\epsilon_1}^1r^{\al}\,d\al \ge \ep_1\delta_1 \int_0^1 r^\al \,d\al.
 \end{split}
\end{equation*}
Similarly, we may find a small $\ep_2>0$ such that $\min_{\al\in[0,\ep_2]} \mu(\al) \ge \delta_2>0$
and then for all $r< 1$
\begin{equation*}
\int_{0}^1 r^{\al}\mu(\al)\,d\al \ge \ep_2\delta_2\int_0^1 r^\al \,d\al.
\end{equation*}
Hence for $\fy \in (\theta-\pi,\pi-\theta)$, we get for $c_1=\min(\ep_1\delta_1,\ep_2\delta_2)$
\begin{equation*}
 \begin{split}
|zw(z)| &\ge \mathrm{\Re}(zw(z))% = \int_0^1 r^{\al} \cos(\al\fy) \mu(\al) \, d\al
{ \geq \cos(\pi-\theta) \int_0^1 r^\al \mu (\al)\, d\al} \ge c_1 \cos(\pi-\theta) \int_0^1 r^\al\,d\al.
 \end{split}
\end{equation*}
Now it suffices to consider the case $\fy \in [\pi-\theta,\theta]$, and the case $\fy\in [-\theta,\theta-\pi]$
follows analogously. From \eqref{eqn:ep0}, we deduce
\begin{equation*}
|zw(z)| \ge \mathrm{\Re}(zw(z)) \ge \frac{\delta_0}{2} \int_0^1 r^\al\,d\al\quad \forall r\le r_0=\left( \frac{\delta_0}{2(\delta_0+\| \mu \|_{C[0,1]})} \right)^{1/\ep_0}.
\end{equation*}
Then a similar argument for deriving \eqref{eqn:rger0} shows that the inequality \eqref{eqn:key2} holds for $r\ge r_0$
and $\fy \in [\pi-\theta,\theta]$, thereby showing \eqref{eqn:key2}. The
inequality \eqref{eqn:key3} follows from
\begin{equation*}
  {\| \mu  \|_{C[0,1]}}\int_0^1  r^\al \,d\al \ge\int_0^1 r^\al \mu(\al) \,d\al=|z|w(|z|)
\end{equation*}
and the trivial inequality $|zw(z)|\leq |z|w(|z|)$.
\end{proof}

Now we give the main result of this section, namely, stability of
problem \eqref{eqn:fde} with $f\equiv 0$.
\begin{theorem}\label{thm:reg}
Let $\mu\in C[0,1]$ be a non-negative function with $\mu(0)\mu(1)>0$. Then the solution $u$ to problem \eqref{eqn:fde} with
$f\equiv0$ satisfies the following stability estimates for $t\in(0,T]$ and $\nu=0,1$:
\begin{eqnarray}
&&\|A^\nu S^{(m)}(t)v\|_{L^2(\Om)}\le c_T t^{-m-\nu}\ell_1(t)^{\nu}\|v\|_{L^2(\Om)},\ v\in L^2(\Om),
m\ge 0,\label{eqn:new1}\\
&&\|A^\nu S^{(m)}(t)v\|_{L^2(\Om)}\le c t^{-m+1-\nu} \ell_2(t)^{1-\nu} \|Av\|_{L^2(\Om)},\ v\in D(A), \nu+m\ge 1,\label{eqn:new2}
\end{eqnarray}
where $\ell_1(t)=(\log(2T/t))^{-1}$,  $\ell_2(t)=\log \left(\max(t^{-1},2)\right)$ and $c_T>0$ is a constant
that may depend on $d$, $\Om$, $\mu$, $M$, $m$ and $T$.
\end{theorem}
\begin{proof}
The existence and uniqueness of a weak solution was already shown in \cite{LiLuchkoYamamoto:2014}, and it suffices to show the
stability estimates \eqref{eqn:new1} and \eqref{eqn:new2}. First, by the resolvent estimate \eqref{eqn:resol}, we obtain
the following basic estimate on the kernel $H(z)$
\begin{equation*}
  \|H(z)\| = \|(zw(z)I + A)^{-1}\||w(z)|\leq M/|z| \quad\forall z\in \Sigma'_\theta.
\end{equation*}
Let $t>0$, $\theta \in (\pi/2,\pi)$, $\delta>0$. We choose $\delta=1/t$ and denote for short $\Gamma=\Gamma_{\theta,\delta}.$
First we derive the estimate \eqref{eqn:new1} for $\nu=0$ and $m\ge 0$. By the solution representation \eqref{eqn:interep}, we
deduce
\begin{equation*}
 \begin{aligned}
 \|S^{(m)}(t)\| &=\left\|
	\frac{1}{2\pi \mathrm{i}} \int_{\Gamma} z^m e^{zt} H(z)\, dz
		\right\|
    \le c\int_{\Gamma} |z|^m e^{\Re(z)t} \|H(z)\|\, |dz|\\
	&\le c\left( \int_{1/t}^\infty r^{m-1}e^{rt\cos\theta}  \,dr
  + \int_{-\theta}^{\theta}e^{\cos\psi}t^{-m}\,d\psi\right)\leq ct^{-m}.
	\end{aligned}
\end{equation*}
Next we prove estimate \eqref{eqn:new1} for $\nu=1$ and $m\ge 0$. To this end, we take $\delta=2T/t$ in the contour $\Gamma$.
By applying the operator $A$ to both sides of \eqref{eqn:interep} and differentiating with respect to time $t$ we arrive at
\begin{equation}\label{eqn:Asolop}
  AS^{(m)}(t)=\frac{1}{2\pi\mathrm{i}}\int_\Gamma z^m e^{z t} AH(z) dz.
\end{equation}
Owing to the identity
\begin{equation*}
  \begin{aligned}
   AH(z)&=A(zw(z)I+A)^{-1}w(z)    = (I-zH(z))w(z),
  \end{aligned}
\end{equation*}
it follows from Lemma \ref{lem:mu} that
\begin{equation}\label{eqn:AH}
  \left\|AH(z)\right\|\le c|w(z)|\le c\frac{|z|-1}{|z|\log|z|}\quad \forall z\in\Sigma_{\theta}'.
\end{equation}
Hence we obtain from \eqref{eqn:Asolop}
\begin{equation*}
  \begin{aligned}
    \|AS^{(m)}(t)\|  &\le c\int_{\Gamma} |z|^{m}\frac{|z|-1}{|z|\log |z|} e^{\Re(z)t} \, |dz| \\
    &{\le c \int_{2T/t}^\infty r^{m-1}\frac{r-1}{\log r}e^{rt\cos\theta}  \,dr
     +  c_Tt^{-m}\frac{2T/t-1}{\log (2T/t)} \int_{-\theta}^{\theta}e^{2T\cos\psi}\,d\psi}=:I+II.%\leq c t^{-m-1+\al}}.
  \end{aligned}
\end{equation*}
Since $2T/t \ge 2$, we can bound the first term $I$ by
\begin{equation*}
I \le c\int_{2T/t}^\infty r^{m}({\log r})^{-1}e^{rt\cos\theta}  \,dr
 \le c \ell_1(t) \int_{2T/t}^\infty r^{m}e^{rt\cos\theta}  \,dr \le c_Tt^{-m-1}\ell_1(t).
\end{equation*}
Meanwhile the second term $ II$ can be bounded by
\begin{equation*}
 II = c_Tt^{-m}({2T/t-1})/{\log (2T/t)}\le c_Tt^{-m}({2T/t})/{\log (2T/t)}=c_Tt^{-m-1}\ell_1(t).
\end{equation*}
This shows the first estimate \eqref{eqn:new1}. To prove the second estimate \eqref{eqn:new2} with $\nu=0$,
we choose $\delta=1/t$ and denote again $\Gamma=\Gamma_{\theta,\delta}$. Then
\begin{equation*}
  \begin{aligned}
    S^{(m)}(t)v & =\frac{1}{2\pi\mathrm{i}}\int_{\Gamma} z^m e^{zt} H(z) v \,dz  = \frac{1}{2\pi\mathrm{i}}\int_{\Gamma} z^{m-1}e^{zt} zA^{-1}H(z) Av dz.
  \end{aligned}
\end{equation*}
Upon noting the identity
\begin{equation*}
zA^{-1}H(z) = zw(z)A^{-1}(zw(z)I+A)^{-1}= A^{-1}-(zw(z)I+A)^{-1}
\end{equation*}
and the fact that $\int_{\Gamma} z^{m-1}e^{zt}\,dz=0$ for $m\ge 1$,
we have
\begin{equation*}
   \begin{aligned}
      S^{(m)}(t)v & = \frac{1}{2\pi\mathrm{i}}\int_{\Gamma}z^{m-1} e^{zt} v\, dz
      - \frac{1}{2\pi\mathrm{ i}}\int_{\Gamma} z^{m-1}e^{zt}(zw(z)I+A)^{-1} \,dzA v\\
      &=-\frac{1}{2\pi\mathrm{i}}\int_{\Gamma} z^{m-1}e^{zt}(zw(z)I+A)^{-1} \,dzA v.\\
   \end{aligned}
\end{equation*}
By \eqref{eqn:resol} and Lemma \ref{lem:key2} we obtain
\begin{equation*}
  \|(zw(z)I+A)^{-1}\| \le M |zw(z)|^{-1} \leq c\frac{\log|z|}{|z|-1},
\end{equation*}
and thus using this estimate and the the monotone decreasing property of the function $f(x)=-\frac{\log(x)}{1-x}$ on the
positive real axis $\mathbb{R}^+$, we get
\begin{equation*}
  \begin{aligned}
    \|S^{(m)}(t)v\|_{L^2(\Om)} & \le c\left(\int_{\Gamma} |z|^{m-1}e^{\Re(z)t}\|(zw(z)I+A)^{-1}\| \,|dz|\right)\| Av \|_{L^2(\Om)}\\
		&\le {c\left(\int_{1/t}^\infty e^{rt\cos\theta}r^{m-1}\frac{\log r}{r-1} \,dr
    + t^{-m}\frac{\log(1/t)}{1/t-1}\int_{-\theta}^{\theta} e^{\cos\psi} \,d\psi \right)
    \| Av \|_{L^2(\Om)}}\\
    &\le c t^{-m+1} \frac{\log (t^{-1})}{1-t}\| Av \|_{L^2(\Om)}.
  \end{aligned}
\end{equation*}
We observe that if $t_n^{-1}\ge2$, {\it i.e.} $t_n\le 1/2$, then
$\frac{\log(t_n^{-1})}{1-t_n}\le 2\log(t_n^{-1}).$
Otherwise if $t_n^{-1}<2$, {\it i.e.} $t_n\ge1/2$, then by the monotonicity of the
function $f(x)=\frac{\log(x)}{1-x}$ on $\mathbb{R}^+$,
$\frac{\log(t_n^{-1})}{1-t_n}= \frac{\log(t_n)}{t_n-1}\le 2\log(2).$
Then we deduce
\begin{equation*}
   \|S^{(m)}(t)v\|_{L^2(\Om)} \le c t^{-m+1} \ell_2(t)\| Av \|_{L^2(\Om)}. %, \ \ t\in(0,T].
\end{equation*}
Lastly, note that (\ref{eqn:new2}) with $\nu=1$ is equivalent to (\ref{eqn:new1}) with $\nu=0$ and $v$ replaced by $Av$.
This completes the proof of the theorem.
\end{proof}
\begin{remark}\label{rem:shorttime}
The a priori estimate of the solution at short time is given in Theorem \ref{thm:reg}, in which the constant $c_T$
depends on the final time $T$ (see also \cite[Theorem 2.2]{LiLuchkoYamamoto:2014} for the special case $v\in L^2\II$, $\nu=1$ and $m=0$).
The long time asymptotic behavior of the solution in case of $v\in D(A)$ was given in \cite[Theorem 2.1]{LiLuchkoYamamoto:2014},
i.e., it decays like $(\log t)^{-1}$ as $t \rightarrow \infty$; see also \cite[example 6.5]{VergaraZacher:2015}
for related discussions on asymptotic decay.
\end{remark}

\section{Semidiscrete discretization by Galerkin FEM}\label{sec:semi}
Now we discuss the space semidiscrete scheme \eqref{eqn:semidiscrete} based on the Galerkin FEM.
On the finite element space $X_h$, we define the $L^2(\Om)$-orthogonal projection $P_h:L^2(\Om)\to X_h$ and
the Ritz projection $R_h:H^1_0(\Om)\to X_h$, respectively, by
\begin{equation*}
  \begin{aligned}
    (P_h \fy,\chi) & =(\fy,\chi) \quad\forall \chi\in X_h,\\
    (\nabla R_h \fy,\nabla\chi) & =(\nabla \fy,\nabla\chi) \quad \forall \chi\in X_h.
  \end{aligned}
\end{equation*}

The Ritz projection $R_h$ and the $L^2(\Om)$-projection $P_h$ have the following properties \cite{Thomee:2006}.
\begin{lemma}\label{lem:prh-bound}
Let the mesh $\mathcal{T}_h$ be quasi-uniform. Then the operators $R_h$ and $P_h$ satisfy:
\begin{eqnarray*}
  &   \| \fy-R_h \fy\|_{L^2(\Om)}+h\|\nabla( \fy-R_h \fy)\|_{L^2(\Om)}\le ch^q\| \fy\|_{ H^q(\Om)}\quad
   \forall\fy \in H_0^1\II\cap H^q(\Om), \ q=1,2,\\
  & \| \fy-P_h \fy\|_{L^2(\Om)}+h\|\nabla( \fy-P_h \fy)\|_{L^2(\Om)}\le ch^q\| \fy\|_{ H^q(\Om)}\quad
   \forall\fy \in H_0^1\II\cap H^q(\Om), \ q=1,2.
\end{eqnarray*}
In addition, $P_h$ is stable on $H_0^{q}\II$ for $0\le q \le1$.
\end{lemma}

The space semidiscrete Galerkin scheme for problem \eqref{eqn:fde} reads: find $u_h(t)\in X_h$ such that
\begin{equation}\label{eqn:fem}
  (\Dmu u_h,\chi) + (\nabla u_h,\nabla \chi) = (f,\chi)\quad\forall \chi\in X_h,
\end{equation}
with $u_h(0)=v_h\in X_h$. Upon introducing the discrete Laplacian $\Delta_h: X_h\to X_h$ defined by
\begin{equation*}
  -(\Delta_h\fy,\chi)=(\nabla\fy,\nabla\chi)\quad\forall\fy,\,\chi\in X_h,
\end{equation*}
and $A_h=-\Delta_h$, the space semidiscrete Galerkin scheme \eqref{eqn:fem} can be rewritten as
\begin{equation}\label{eqn:fdesemidis}
  \Dmu u_h(t) +A_h u_h(t) = 0, \,\, t>0
\end{equation}
with $u_h(0)=v_h\in X_h$ and $A_h=-\Delta_h$.

For the error analysis of the semidiscrete scheme \eqref{eqn:fdesemidis}, we employ an operator trick due to Fujita and Suzuki \cite{FujitaSuzuki:1991}.
To this end, we first represent the semidiscrete solution $u_h$ to \eqref{eqn:fdesemidis} by
\begin{equation}\label{eqn:semi-interep}
    u_h(t)= S_h(t) v_h := \frac{1}{2\pi \mathrm{i}} \int_{\Gamma_{\theta,\delta}} e^{zt}(zw(z)I+A_h)^{-1}w(z) v_h\,dz.
\end{equation}

\begin{lemma}\label{lem:keylem}
For any $ \fyy \in H_0^1(\Om)$ and $z\in \Sigma_{\theta,\delta}$ for $\theta\in(\pi/2,\pi)$, there holds
\begin{equation}%\label{eqn:key}
    |zw(z)| \| \fyy \|_{L^2(\Om)}^2 + \| \nabla \fyy \|_{L^2(\Om)}^2
    \le c\left|zw(z)\| \fyy \|_{L^2(\Om)}^2 + \|\nabla\fyy \|^2 \right|.
\end{equation}
\end{lemma}
\begin{proof}
With Lemma \ref{lem:sector-w}, the proof is identical to that of \cite[Lemma 3.3]{BazhlekovaJinLazarovZhou:2014}, and hence omitted.
\end{proof}

Now we introduce the error function $e(t):=u(t)-u_h(t)$  which, in view of \eqref{eqn:interep} and  \eqref{eqn:semi-interep}, can be represented by \begin{equation}\label{error-e}
e(t)=\frac1{2\pi\mathrm{i}}\int_{\Gamma_{\theta,\delta}} e^{zt} w(z) (\hfy(z) - \hfy_h(z) ) \,dz,
\end{equation}
with $\hfy(z)=(zw(z)I+A)^{-1}v$ and $\hfy_h(z)=(zw(z)I+A_h)^{-1}P_h v$.
The following lemma shows a bound on the error $\hfy_h-\hfy $. It follows directly from Lemma \ref{lem:keylem},
similar to \cite[Lemma 3.4]{BazhlekovaJinLazarovZhou:2014}, and hence the proof is omitted.
\begin{lemma}\label{lem:wbound}
Let $ v\in L^2(\Om) $,  $z\in \Sigma_{\theta}$ with $\theta\in(\pi/2,\pi)$,
 $\hfy(z)=(zw(z)I+A)^{-1}v$ and $\hfy_h(z)=(zw(z)I+A_h)^{-1}P_h v$. Then there holds
\begin{equation}\label{eqn:wboundHa}
    \|  \hfy(z) - \hfy_h(z) \|_{L^2(\Om)} + h\| \nabla (\hfy(z) - \hfy_h(z) ) \|_{L^2(\Om)}
    \le ch^2 \| v  \|_{L^2(\Om)}.
\end{equation}
\end{lemma}

Now we can state an error estimate for nonsmooth initial data $v\in L^2\II$.
\begin{theorem}\label{thm:nonsmooth-initial-op}
Let $u$ and $u_h$ be the solutions of problem \eqref{eqn:fde} and \eqref{eqn:fdesemidis} with $v \in L^2\II$
and $v_h=P_h v$, respectively. Then for $t>0$ and $\ell_1(t)=\log(2T/t)^{-1}$, there holds
\begin{equation*}
  \| u(t)-u_h(t) \|_{L^2\II} + h\| \nabla (u(t)-u_h(t)) \|_{L^2\II}
  \le c_Th^{2}t^{-1} \ell_1(t) \|v\|_{L^2\II}.
\end{equation*}
\end{theorem}
\begin{proof}
In the error representation \eqref{error-e}, by choosing $\delta=2T/t$ in the
contour $\Gamma_{\theta,\delta}$ and appealing to Lemmas \ref{lem:wbound} and \ref{lem:mu}, we deduce
\begin{equation*}
 \begin{split}
  \| \nabla e(t)  \|_{L^2\II} &\le ch \int_{2T/t}^\infty e^{rt\cos\theta} \frac{r-1}{r\log r} dr \| v \|_{L^2\II}
+ ch \int_{-\theta}^{\theta} e^{2T\cos\psi} \frac{2T/t-1}{\log (2T/t)} d\psi \| v \|_{L^2\II} :=I+II.
 \end{split}
\end{equation*}
Now the first term $I$ can be bounded by
\begin{equation*}
  \begin{aligned}
   I \le  ch \int_{2T/t}^\infty e^{rt\cos\theta} \frac{1}{\log r} dr \| v \|_{L^2\II}
    \le \frac{ch}{\log (2T/t)} \int_{2T/t}^\infty e^{rt\cos\theta}dr \le c_Th t^{-1} \ell_1(t) \| v  \|_{L^2\II},
  \end{aligned}
\end{equation*}
and the second term $II$ is bounded by
\begin{equation*}
  II \le \frac{c_Th}{t\log (2T/t)} \int_{-\theta}^{\theta} e^{2T\cos\psi} d\psi \| v \|_{L^2\II} \le c_Th t^{-1} \ell_1(t) \| v \|_{L^2\II}.
\end{equation*}
The bound on $\| \nabla e(t) \|_{L^2\II}$ now follows by the triangle inequality. A similar
argument yields the desired $L^2\II$ error estimate.
\end{proof}

Next we turn to the case of smooth initial data, {\it i.e.}, $Av\in L^2(\Om)$, and derive the following error estimate.
\begin{theorem}\label{thm:smooth-initial-op}
Let $u$ and $u_h$ be the solutions of problem \eqref{eqn:fde} and \eqref{eqn:fdesemidis} with $v \in \dH2$
and $v_h=R_h v$, respectively. Then for $t>0$, there holds:
\begin{equation}\label{eqn:smooth-initial-op}
  \| u(t)-u_h(t) \|_{L^2\II} + h\| \nabla (u(t)-u_h(t)) \|_{L^2\II} \le ch^{2} \|Av\|_{L^2(\Om)}.
\end{equation}
\end{theorem}
\begin{proof}
Like before, we take $\theta \in (\pi/2,\pi)$ and $\delta=1/t$ in the contour $\Gamma_{\theta,\delta}$. Then the error $e_h(t)=u(t)-u_h(t)$ can be represented by
\begin{equation*}
e_h(t)=\frac1{2\pi\mathrm{i}}\int_{\Gamma_{\theta,\delta}} e^{zt} w(z) \left((zw(z)I+A)^{-1}-(zw(z)I+A_h)^{-1}R_h\right)v \,dz.
\end{equation*}
Using the identity
\begin{equation*}
  w(z) (zw(z)I+A)^{-1} = z^{-1}I-z^{-1}(zw(z)I+A)^{-1}A,
\end{equation*}
we deduce
\begin{equation}\label{eqn:smooth-er-rep}
  \begin{split}
    e_h(t)=\frac1{2\pi\mathrm{i}}\left(\int_{\Gamma_{\theta,\delta}} e^{zt} z^{-1} (\hfy_h(z)-\hfy(z)) \,dz
    + \int_{\Gamma_{\theta,\delta}} e^{zt} z^{-1} (v-R_hv) \,dz\right),
  \end{split}
\end{equation}
where $\hfy(z)=(zw(z)I+A)^{-1}A v$ and $\hfy_h(z)=(zw(z)I+A_h)^{-1}A_hR_hv$.
Then Lemmas \ref{lem:prh-bound} and \ref{lem:wbound}, and the identity $A_hR_h=P_h A$ give
\begin{equation*}
  \| \hfy(z)-\hfy_h(z) \|_{L^2\II} + h\| \nabla (\hfy(z)-\hfy_h(z)) \|_{L^2\II}
  \le ch^{2} \| Av \|_{L^2\II}.
\end{equation*}
Now it follows from this and the representation \eqref{eqn:smooth-er-rep} that
\begin{equation*}
\begin{split}
  \|e_h(t)\| &\le  ch^2 \| Av \|_{L^2\II}\left(\int_{1/t}^\infty e^{rt\cos\theta} r^{-1} \,dr
      + \int_{-\theta}^{\theta} e^{\cos\psi} \,d\psi \right)\le c h^2 \| Av \|_{L^2\II},
\end{split}
\end{equation*}
which gives the $L^2\II$-error estimate. The $H^1\II$ estimate follows analogously.
\end{proof}

\begin{remark}
The error estimate for nonsmooth initial data $v\in L^2\II$ deteriorates like $t^{-1}\ell_1(t)$ as $t\to0^+$. The behavior
agrees with the solution singularity in Theorem \ref{thm:reg}. The factor $t^{-1}\ell_1(t)$ is different from that for subdiffusion
\cite{JinLazarovZhou:2013} and multi-term time fractional diffusion \cite{JinLazarovLiuZhou:2014}. In contrast, for smooth
initial data $Av\in L^2(\Om)$, the error estimate is uniform in $t$.
\end{remark}

\section{Fully discrete scheme I: Laplace transform}\label{sec:Laplace}
The first fully discrete scheme is based on the Laplace transform. To this end, we select a proper contour $\Gamma_{
\theta,\delta}$ in the integral representation \eqref{eqn:semi-interep} of the semidiscrete solution $u_h$, and then apply a quadrature
rule. We follow the works \cite{gavrilyuk2002mathcal,%lopez2004numerical,
Lopez-FernandezPalencia:2006,mclean2010numerical,SheenSloanThomee:2000,SheenSloanThomee:2003,WeidemanTrefethen:2007}
and deform the contour $\G_{\theta,\delta}$ to
be a curve with the following parametric representation
\begin{equation}\label{eqn:curve}
 z(\xi):= \lambda(1+\sin(\mathrm{i} \xi-\psi )),
\end{equation}
with $\lambda>0$, $\psi\in(0,\pi/2)$ and $\xi \in \mathbb{R}$. The optimal choices of $\lambda$ and $\psi$ will be given
below, in the proof of Lemma \ref{lem:quad-error}. This deformation is valid since it does not transverse
the poles of the kernel function $H(z)v=(zw(z)+A_h)^{-1}w(z)v$, \emph{cf.}, Lemma \ref{lem:sector-w} and Lemma \ref{lem:quad-error2}
below. Upon letting $z=x+\mathrm{i} y$, we deduce that the contour \eqref{eqn:curve} is the left branch of the hyperbola
\begin{equation}\label{eqn:hyper}
 \left(\frac{x-\la}{\la\sin\psi}\right)^2 -  \left(\frac{y}{\la\cos\psi}\right)^2 = 1,
\end{equation}
which intersects the real axis at $x=\la(1\pm\sin\psi)$ and has asymptotes $ y=\pm (\lambda-x)\cot\psi. $
Now we can represent the semidiscrete solution $u_h(t)$ by
\begin{equation}\label{eqn:interep2}
  u_h(t) = \int_{-\infty}^\infty {\hg}(\xi,t) \,d\xi
\end{equation}
with the integrand $\hg(\xi,t)$ being defined by
\begin{equation}\label{eqn:interep3}
{\hg}(\xi,t)=\frac{1}{2\pi\mathrm{i}}e^{z(\xi)t}\left(z(\xi)w(z(\xi))I+A_h\right)^{-1}w(z(\xi))z'(\xi)v_h.
\end{equation}

\begin{remark}
The integrand ${\hg}(\xi,t)$ exhibits a double
exponential decay as $|\xi|\to\infty$ for $t>0$.
\end{remark}

Now we describe the quadrature rule for approximating \eqref{eqn:interep2}. By setting $z_j=z(\xi_j)$ and $z_j':=z'(\xi_j)$
with $\xi_j=jk$ and $k$ being the step size, we have the following quadrature approximation
\begin{equation}\label{eqn:quad0}
 U_{h}(t)= \frac{k}{2\pi\mathrm{i}}\sum_{j=-\infty}^{\infty}e^{z_jt} {\hphi}_j z'_jv_h,
\end{equation}
and the truncated quadrature approximation
\begin{equation}\label{eqn:trunc-quad}
 U_{N,h}(t)= \frac{k}{2\pi\mathrm{i}}\sum_{j=-N}^{N}e^{z_jt} {\hphi}_j z'_j,
\end{equation}
with ${\hphi}_j = \left(z_jw(z_j)I+A_h\right)^{-1}w(z_j) v_h$. To compute $U_{N,h}(t)$, we need to solve
only $N+1$ elliptic problems, instead of $2N+1$ elliptic problems, by exploiting the conjugacy relations: $z_{-j} = \overline{z_j}$, $w(z_{-j}) =
\overline{w(z_{-j})}$, ${\hphi}_{-j} = \overline{{\hphi}_j}$, $j = 1, \cdots, N.$
Indeed, since $z'_j = z'(\xi_j) = \mathrm{i}\lambda \cos(\mathrm{i}\xi_j-\psi),$
denoting by $\zeta_j = \lambda \cos(\mathrm{i}\xi_j-\psi)$,
\eqref{eqn:trunc-quad} is reduced to
\begin{equation}\label{eqn:trun-quad-symm}
 U_{N,h}(t)= \frac{k}{\pi}\left[\frac12 e^{z_0t} {\hphi}_0 \zeta_0 + \sum_{j=1}^{N}\Re\{e^{z_jt} {\hphi}_j \zeta_j\}\right],
\end{equation}
Hence we solve the following complex--valued elliptic problems
\begin{equation}\label{eqn:ellip}
 \left(z_jw(z_j)I+A_h\right){\hphi}_j = w(z_j) v_h,\quad j=0,\ldots,N.
\end{equation}
These problems are independent of each other and can be solved in parallel, if desired.

Next, we define a strip $\cS_{a,b}\subset \mathbb C$ by
\begin{equation*}
  \cS_{a,b} = \{ p=\xi+\mathrm{i} \eta:~~\text{for all}~~\xi\in\mathbb{R}~~\text{and}~~\eta\in(-b,a)  \}.
\end{equation*}
The following lemma
recalls a known error estimate for the quadrature \cite{Martensen:1968} \cite[Theorem 2.1]{WeidemanTrefethen:2007}.
The quadrature is exponentially convergent, provided that the integrand $g$ is analytic on
a strip $\cS_{a,b}$ with some additional conditions.
\begin{lemma}\label{lem:keylem-quad}
Let $g$ be an analytic function in a strip $\cS_{a,b}$ for some $a,b>0$, and $I$ and $I_k$, for $k>0,$ be defined by
\begin{equation*}
I=\int_{-\infty}^\infty g(x) \,dx\quad\mbox{and} \quad I_k= k\sum_{j=-\infty}^\infty g(jk),
\end{equation*}
respectively. Furthermore, assume that $g(z)\rightarrow 0$ uniformly as $|z|\rightarrow \infty$ in the strip $\cS_{a,b},$
and that there exist $M_+>0$ and $M_->0$, which may depend on $a$ and $b$ such that
\begin{equation*}
   \lim_{r\rightarrow a^-} \int_{-\infty}^\infty|g(x+\mathrm{i} r)| \,dx\le M_+, \quad \lim_{s\rightarrow b^-}\int_{-\infty}^\infty |g(x-\mathrm{i} s)|\,dx\le M_-.
\end{equation*}
Then the approximation error can be bounded by
\begin{equation*}
    |I-I_k|\le E^++E^-,
\end{equation*}
where
\begin{equation*}
    E^+=\frac{M_+}{e^{2a\pi/k}-1} \quad \text{and} \quad E^-=\frac{M_-}{e^{2b\pi/k}-1}.
\end{equation*}
\end{lemma}

The next lemma gives one crucial estimate on the map $z(p)$ over the strip $\mathcal{S}_{a,b}$. Even though
the hyperbolic contour \eqref{eqn:curve} has been extensively used, the estimate on the map $z(p)$ below
seems to be new and it is of independent interest.
\begin{lemma}\label{lem:pp}
Let $p=\xi+\mathrm{i}\eta$ with $\xi,~~\eta\in\mathbb{R}$. Then with $a=\pi/2-\psi-\ep$ and $b=\psi-\ep$, for small $\ep>0$, there holds
\begin{eqnarray}
&& z(p)\in \Sigma_{\pi-\psi}\quad \mbox{and}\quad  \left|\frac{z'(p)}{z(p)}\right|\leq \frac{c}{\ep}\qquad \forall p\in\overline{\mathcal{S}}_{a,0},\label{eqn:pp1}\\
&& z(p)\in \Sigma_{\pi-\ep}\quad \mbox{and}\quad  \left|\frac{z'(p)}{z(p)}\right|\leq c\qquad \forall p\in\overline{\mathcal{S}}_{0,b}.\label{eqn:pp2}
\end{eqnarray}
\end{lemma}
\begin{proof}
For $p=\xi+\mathrm{i}\eta$ with $\xi,~~\eta\in\mathbb{R}$, then the image $z(p)$ in the parameterization \eqref{eqn:curve} is given by
\begin{equation*}
 z(p)=\lambda(1-\sin(\psi+\eta)\cosh(\xi)) + \mathrm{i}\lambda\cos(\psi+\eta)\sinh(\xi),
\end{equation*}
and its derivative $z'(p)$ is given by
\begin{equation*}
  z'(p)= \lambda \cosh\xi\cos(\psi+\eta)-\mathrm{i}\sinh\xi\sin(\psi+\eta).
\end{equation*}
By writing $z=x+\mathrm{i} y$, it can be expressed as the left branch of the hyperbola
\begin{equation*}
   \left(\frac{x-\lambda}{\lambda\sin(\psi+\eta)}\right)^2 - \left(\frac{y}{\cos(\psi+\eta)}\right)^2 = 1.
\end{equation*}
It intersects the real axis at $x=\la(1-\sin(\psi+\eta))$ and has the asymptotes $y=\pm(x-\la)\cot(\psi+\eta)$.
Next we show the estimates \eqref{eqn:pp1} and \eqref{eqn:pp2}.
First, for $p \in \overline{\mathcal{S}}_{a,0}$, \textit{i.e.}, $\eta \in [0,a]$, $z(p)$ lies in the sector $\Sigma_{\pi-\psi}$.
Using the elementary identity $\sinh^2x=\cosh^2x-1$, the fact $\fy:=\eta+\psi\in (\psi,\pi/2-\ep)$, and the estimate $\sin(
\pi/2-\ep)\sim 1-\ep^2/2\le 1-\ep^2/3$ for small $\ep$, we have for all $\xi \in \mathbb{R}$
\begin{equation*}
\begin{split}
  \bigg|\frac{z'(p)}{z(p)}\bigg|^2 &= \bigg|\frac{\cos(\fy)\cosh(\xi)-\mathrm{i}\sin(\fy)\sinh(\xi)}{(1-\cosh(\xi)\sin(\fy))+\mathrm{i}\sinh(\xi)\cos(\fy)}\bigg|^2\\
 & = \frac{\cos^2(\fy)\cosh^2(\xi)+\sin^2(\fy)\sinh^2(\xi)}{1-2\cosh(\xi)\sin(\fy)+\cosh^2(\xi)\sin^2(\fy)+\sinh^2(\xi)\cos^2(\fy)}
  = \frac{\cosh^2(\xi)-\sin^2(\fy)}{(\cosh(\xi)-\sin(\fy))^2} \\
  & = \frac{\cosh(\xi)+\sin(\fy)}{\cosh(\xi)-\sin(\fy)} \le \frac{1+\sin(\fy)}{1-\sin(\fy)} \le\frac{2}{1-(1-\ep^2/3)}\le \frac{6}{\ep^2}.
\end{split}
\end{equation*}
Hence the estimate \eqref{eqn:pp1} holds true. Now we turn to the case $p \in \overline{\mathcal{S}}_{0,b}$, \textit{i.e.}, $\eta
\in [-b,0]$. Then $z(p)$ lies in the sector $\Sigma_{\pi-(\eta+\psi)}\subset\Sigma_{\pi-\ep}$. Further, by noting $\fy:=
\eta+\psi\in (\ep,\psi)$, we have for all $\xi \in \mathbb{R}$
\begin{equation*}
  \bigg|\frac{z'(p)}{z(p)}\bigg|^2   \le \frac{1+\sin(\fy)}{1-\sin(\fy)} \le\frac{1+\sin(\psi)}{1-\sin(\psi)}.
\end{equation*}
Then the desired result \eqref{eqn:pp2} follows directly.
\end{proof}

The next result gives the analyticity of and an estimate on the integrand
${\hg}(\xi,t)$ on the strip $\cS_{a,b}$.
\begin{lemma}\label{lem:quad-error2}
Let $p=\xi+\mathrm{i}\eta$ with $\xi,~~\eta\in\mathbb{R}$ and ${\hg}(p,t)$ be defined by \eqref{eqn:interep3}.
Then $\hg(p,t)$ is analytic on the strip $\cS_{a,b}$, and the following estimate holds:
\begin{equation*}
 \| {\hg}(p,t)\| \le \frac{c}{\ep} e^{\lambda (1-\sin(\psi+\eta)\cosh(\xi) ) t}\| v_h
  \|_{L^2\II}\quad\forall p \in \cS_{a,b}.
\end{equation*}
\end{lemma}
\begin{proof}
For $p=\xi+\mathrm{i}\eta$ with $\xi,~~\eta\in\mathbb{R}$, the image $z(p)$ in \eqref{eqn:curve} is given by
\begin{equation*}
 z(p)=\lambda(1-\sin(\psi+\eta)\cosh(\xi)) + \mathrm{i}\lambda\cos(\psi+\eta)\sinh(\xi).
\end{equation*}
By Lemmas \ref{lem:pp} and \ref{lem:sector-w}, and Remark \ref{rem:eparg}, $z(p)w(z(p))\in \Sigma_{\pi-\ep'}$, with $\ep'>0$.
Hence the function
\begin{equation*}
  {\hg}(p,t)=\frac{1}{2\pi\mathrm{i}}e^{z(p)t}\left(z(p)w(z(p))I+A_h\right)^{-1}w(z(p))z'(p)v_h
\end{equation*}
is analytic in the strip ${\cS}_{a,b}$. It remains to show the estimate. First, we consider the case $p\in
\overline{\cS}_{0,b}$. By \eqref{eqn:pp2}, $z(p)\in \Sigma_{\pi-\ep}$. Then, by Lemma \ref{lem:sector-w}
and Remark \ref{rem:eparg}, $z(p)w(z(p))\in \Sigma_{\pi-\ep'}$, with $\ep'=c\ep$.
By the resolvent estimate \eqref{eqn:resol}, we deduce that for small $\ep>0$, there holds
\begin{equation}\label{eqn:resol-ep}
\| \left(zI+A_h\right)^{-1}\| \le c/|\Im(z)|\le c/|z\sin(\pi-\ep)|\le c/(|z|\ep')\quad \forall z\in \Sigma_{\pi-\ep'}'.
\end{equation}
Meanwhile, for any $p \in \overline{\mathcal{S}}_{0,b}$, there holds
\begin{equation*}
  \Re (z(p)) = \lambda(1-\sin(\psi+\eta)\cosh(\xi)),
\end{equation*}
which together with the resolvent estimate \eqref{eqn:resol-ep} and Lemma \ref{lem:sector-w} yields
\begin{equation*}
 \begin{split}
  \|{\hg}(p,t)\|   &\le  c e^{\Re(z(p))t} | z'(p)w(z(p))| ~~\| (z(p)w(z(p))+A)^{-1}  \|  ~~ \| v_h  \|_{L^2\II}\\
  &\le \frac{c}{\ep}e^{\lambda(1-\sin(\psi+\eta)\cosh(\xi))t} \bigg|\frac{z'(p)}{z(p)}\bigg|  \| v_h  \|_{L^2\II}.
 \end{split}
\end{equation*}
This together with \eqref{eqn:pp2} yields the desired assertion. The case $p\in \overline{\cS}_{a,0}$ is more direct.
Then \eqref{eqn:pp2} and Lemma \ref{lem:sector-w} imply that $z(p)w(z(p)) \in \Sigma_{\theta'}$ with $\theta'\in(\pi/2,
\pi)$ depending only on $\psi$. Then the desired assertion follows from \eqref{eqn:pp1} and the resolvent estimate
\eqref{eqn:resol}.
\end{proof}

Now we can give an error estimate for the quadrature approximation $U_{N,h}$.
\begin{lemma}\label{lem:quad-error}
Let $u_h(t)$ and $U_{N,h}(t)$ be defined in \eqref{eqn:interep2} and  \eqref{eqn:trunc-quad}, respectively, and the contour be
parametrically represented by \eqref{eqn:curve}. Then with the choice $k=c_0/N$ and $\lambda=c_1N/t$, there holds
\begin{equation*}
  \| u_h(t) - U_{N,h}(t) \|_{L^2\II}  \le c e^{ -c'N } \|  v \|_{L^2\II},
\end{equation*}
where the constant $c$ and $c'$ depend on the choice of $\psi$ in \eqref{eqn:curve}.
\end{lemma}
\begin{proof}
We use the following splitting
\begin{equation*}
 u_h-U_{N,h}=(u_h-U_h)+(U_h-U_{N,h})=: E_q+E_t,
\end{equation*}
where $E_q$ and $E_t$ denote the quadrature and truncation error, respectively. We apply Lemma \ref{lem:keylem-quad} to bound
$\|E_q\|_{L^2\II}$. To this end, we set $a=\pi/2-\psi-\ep$ and $b=\psi-\ep$. For $p=\xi+\mathrm{i} a$, $ zw(z) $ lies in the
sector $\Sigma_\theta$ for some $\theta\in(\pi/2,\pi)$. Note the elementary inequalities $\cosh \xi \geq 1 + \xi^2/2$ and
$1-\sin(\pi/2-\ep) \leq \ep$ for small $\ep>0$. These together with the choice $\lambda=c_1N/t$  and Lemma \ref{lem:quad-error2} yield
\begin{equation*}
 \begin{split}
 \bigg|\hspace{-0.6mm}\bigg|\int_{-\infty}^\infty |{\hg}(\xi+\mathrm{i} a)|\,d\xi \bigg|\hspace{-0.6mm}\bigg|_{L^2\II}
 &\le \frac{c}{\ep} \int_0^\infty e^{c_1N(1-\sin(\pi/2-\ep)\cosh(\xi))} \,d\xi \| v_h  \|_{L^2\II}\\
 &\le \frac{c}{\ep} e^{c_1N\ep}  \int_0^\infty e^{-c_1N \sin(\pi/2-\ep)\xi^2/2} \,d\xi \| v_h  \|_{L^2\II}\\
 &\le \frac{c}{\ep} N^{-\frac12} e^{c_1N\ep} \| v_h  \|_{L^2\II}.
 \end{split}
\end{equation*}
Using Lemma \ref{lem:keylem-quad}, for $k=c_0/N$ we have
\begin{equation*}
 \| E_q^+\|_{L^2\II} \le \frac{c}{\ep} N^{-\frac12}e^{-(2\pi(\pi/2-\psi-\ep)/c_0-\ep c_1)N}.
\end{equation*}
Next we bound the error due to the lower half. For the choice $p=\xi-\mathrm{i} b$, $\lambda=c_1N/t$ and appealing
again to the inequality $\cosh \xi\geq 1 +\xi^2/2$, we deduce
\begin{equation*}
\begin{split}
  \bigg|\hspace{-0.6mm}\bigg| \int_{-\infty}^\infty|{\hg}(\xi - \mathrm{i} b)|\,d\xi  \bigg|\hspace{-0.6mm}\bigg|_{L^2\II}
  &\le \frac{c}{\ep} \int_0^\infty e^{c_1N(1-\sin(\ep)\cosh(\xi))}\,d\xi~~ \| v_h  \|_{L^2\II} \\
  &\le \frac{c}{\ep} e^{c_1N(1-\sin(\ep))}\int_0^\infty e^{-{ c_1 N}\sin(\ep)\xi^2/2}\,d\xi ~~\| v_h  \|_{L^2\II} \\
  &\le \frac{c}{\ep^{3/2}} N^{-\frac12} e^{c_1N(1-\ep)} \| v_h  \|_{L^2\II}.
\end{split}
\end{equation*}
Then for the choice $k=c_0/N$, Lemma \ref{lem:keylem-quad} yields the following estimate
\begin{equation*}
    \| E_q^-\|_{L^2\II} \le\frac{c}{\ep^{3/2}} N^{-\frac12}e^{-(2\pi(\psi-\ep)/c_0-c_1(1-\ep))N}.
\end{equation*}
Further, by using $\cosh(\xi) \ge \cosh (c_0) + \sinh (c_0)
(\xi-c_0)$ for $\xi \ge c_0,$
the truncation error $\|E_t\|_{L^2\II}$ can be simply estimated by
\begin{equation*}
\begin{split}
  \| E_t\|_{L^2\II} &\le \frac{c}{\ep}\int_{c_0}^\infty e^{c_1N
    (1-\sin(\psi)\cosh(\xi) )} \,d\xi \|v_h\|_{L^2\II} \\
  &\le \frac{c}{\ep} e^{c_1N (1-\sin(\psi)\cosh(c_0))}
  \int_{c_0}^\infty e^{- c_1N \sin(\psi) \sinh(c_0) (\xi-c_0)} \,d\xi \|v_h\|_{L^2\II}\\
  &= \frac{c}{c_1 \sin(\psi) \sinh(c_0) \ep} N^{-1}e^{c_1N [1-\sin(\psi) \cosh(c_0)]} \|v_h\|_{L^2\II}.
\end{split}
\end{equation*}
Finally, by disregarding $\ep$ terms, balancing asymptotically the exponential parts in $\| E_q^+\|_{L^2\II}
$, $\| E_q^-\|_{L^2\II} $ and $\| E_t^-\|_{L^2\II} $, we arrive at
\begin{equation*}
  2\pi(\pi/2-\psi)/c_0= 2\pi\psi/c_0-c_1 =-c_1(1-\sin(\psi)\cosh(c_0)).
\end{equation*}
We may express the parameters $c_0$ and $c_1$ in terms of $\psi$:
$$ c_0= \cosh^{-1}\left(  \frac{2\pi\psi}{(4\pi\psi-\pi^2)\sin\psi} \right)\quad \mbox{and}\quad c_1=(4\pi\psi-\pi^2)/ \cosh^{-1}\left(  \frac{2\pi\psi}{(4\pi\psi-\pi^2)\sin\psi} \right).$$
Finally we minimize the ratio
$$B(\psi)=  c_1-2\pi\psi/c_0$$
with respect to the parameter $\psi$, which achieves the minimum at $\psi=1.1721$ and hence,
\begin{equation*}
 c_0= 1.0818,\quad c_1=4.4920\quad \text{and}\quad B(\psi)=-2.32,
\end{equation*}
which are identical to those values given in \cite{WeidemanTrefethen:2007}. Then collecting the balanced asymptotic
bound and the rest from  $\| E_q^+\|_{L^2\II} $, $\| E_q^-\|_{L^2\II} $ and $\| E_t^-\|_{L^2\II} $ yields
\begin{equation*}
   \| u_h(t) - U_{N,h}(t) \|_{L^2\II}  \le c\left(\ep^{-1}N^{-1/2}+\ep^{-3/2}N^{-1/2}+\ep^{-1}N^{-1}\right)
   e^{-[2.32- (2\pi/c_0 + c_1)\ep]N} \|  v_h \|_{L^2\II}.
\end{equation*}
Now by choosing $\ep=1/N$, we get
\begin{equation*}
  \| u_h(t) - U_{N,h}(t) \|_{L^2\II}  \le c e^{(-2.32 + \frac{\log N}{N})N} \|  v_h \|_{L^2\II}.
\end{equation*}
which together with the fact $(\log x)/x\le 1/e$ for $x\ge1$ and the $L^2$-stability of the projection $P_h$ yields the desired result.
\end{proof}

Last, we give the main result of this section, {\it i.e.}, error estimates for the fully discrete scheme \eqref{eqn:trunc-quad}. It follows
from Theorems \ref{thm:nonsmooth-initial-op} and \ref{thm:smooth-initial-op},
and Lemma \ref{lem:quad-error} and the triangle inequality.
\begin{theorem}\label{thm:error_fully}
Let $u(t)$ be the solution of problem \eqref{eqn:fde}, and $U_{N,h}(t)$ be the quadrature approximation defined in \eqref{eqn:trunc-quad}, with the parameters chosen as in Lemma \ref{lem:quad-error}. Then with $\ell_1(t) = (\log 2T/t)^{-1}$, the following estimates hold.
\begin{itemize}
  \item[(a)] If $Av\in L^2\II$ and $v_h=R_h v$, then \begin{equation*}
   \| u(t)-U_{N,h}(t) \|_{L^2(\Om)} \le c\left(e^{-c'N } + h^2\right)  \| Av \|_{L^2\II}.
  \end{equation*}
  \item[(b)] If $v\in L^2(\Om)$ and $v_h=P_hv$, then
  \begin{equation*}
   \| u(t)-U_{N,h} \|_{L^2(\Om)} \le c_T\left(e^{-c'N} + h^2 t^{-1}\ell_1(t)\right)   \| v\|_{L^2\II}.
  \end{equation*}
\end{itemize}
\end{theorem}

\section{fully discrete scheme II: convolution quadrature}\label{sec:Conv-quad}
Now we develop a second fully discrete scheme based on convolution quadrature generated by
the backward Euler method, and show that the scheme is first order convergent.

\subsection{Time stepping based on convolution quadratures}
To describe the fully discrete scheme, we divide the interval $[0,T]$ into a uniform grid with a time step size
$\tau = T/N$, $N\in\mathbb{N}$, with $0=t_0<t_1<\ldots<t_N=T$, and $t_n=n\tau$, $n=0,\ldots,N$. The general
construction of convolution quadrature is as follows \cite{Lubich:1988, CuestaLubichPlencia:2006}. Let $(\sigma,
\rho)$ be a stable and consistent implicit linear multistep method, with $(\sigma,\rho)$ being its
characteristic polynomials. Then we define convolution quadrature weights $\{b_j\}_{j=0}^\infty$ by the
expansion coefficients of
\begin{equation*}
    \widetilde \omega(\xi)= \sum_{j=0}^\infty b_{j}\xi^j=\int_0^1 \left(\frac{\sigma(1/\xi)}{\rho(1/\xi)}\right)^\al \mu(\al) \, d\al.
\end{equation*}
We consider only the simplest case, {\it i.e.}, the backward Euler method, for which the convolution quadrature
weights $\{b_j\}_{j=0}^\infty$ are defined by
\begin{equation}\label{eqn:BE-weight}
    \widetilde \omega(\xi)= \sum_{j=0}^\infty b_{j}\xi^j=\int_0^1 \left( \frac{1-\xi}{\tau}\right)^\al \mu(\al) \, d\al= \left(\frac{1-\xi}{\tau}\right)w\left(\frac{1-\xi}{\tau}\right).
\end{equation}
The convolution quadrature weights $\{b_j\}$ can be computed efficiently using the fast Fourier transform
\cite{Podlubny:1999}, in view of Cauchy's theorem. Then the convolution quadrature $\mathcal{Q}_\tau \fy$
for a Riemann-Liouville fractional derivative $^R \kern -.2em D_t^\al \fy :=\frac{d}{dt}\frac{1}{\Gamma
(1-\alpha)}\int_0^t(t-s)^{-\alpha}\fy(s)ds$ generated by the backward Euler method is given by
\begin{equation}\label{eqn:quad}
  (\mathcal{Q}_\tau \fy) (t_n)= \sum_{j=0}^n b_{n-j} \fy(j\tau).
\end{equation}

Following this general construction, we now derive the time stepping scheme. The approximation $Q_n(\fy)$ to
the Riemann-Liouville fractional derivative ${^R\kern -.2em D_t^\al}\fy(t_n)$ is given by \cite{CuestaLubichPlencia:2006,JinLazarovZhou:2014a}: for any $n=1,2,\ldots,N$:
\begin{equation}\label{quad1}
    Q_n(\fy) = \sum_{j=1}^n b_{n-j}\fy(t_j),
\end{equation}
where the weights $\{b_j\}$ are generated by \eqref{eqn:BE-weight}. Recall also the defining relation of the Caputo derivative using the
Riemann-Liouville derivative \cite[p. 91, equation (2.4.4)]{KilbasSrivastavaTrujillo:2006}
$
   D_t^\al u = {^R\kern -.2em D_t^\al}(u-u(0)).
$
Upon applying the convolution quadrature to the term on the right hand side and using it
for the semidiscrete problem \eqref{eqn:fdesemidis}, we arrive at the following fully discrete scheme for the model
\eqref{eqn:fde}: for $n=1,2,\ldots,N$
\begin{equation}\label{eqn:fully}
   Q_n(U_h) + A_hU_h^n=Q_n(1)v_h,
\end{equation}
with $U_h^0=v_h$. Throughout, we denote the generating function $\widetilde{\beta}$ of a sequence $\{\beta_j\}_{j=0}^\infty$ by $\widetilde{\beta}(\xi)=\sum_{j=0}^\infty \beta_j\xi^j$.
\begin{remark}
Compared with the general construction \eqref{eqn:quad}, the term corresponding to $j=0$ is omitted in our fully discrete scheme \eqref{eqn:fully}.
This choice was taken earlier in \cite{LubichSloanThomee:1996,BazhlekovaJinLazarovZhou:2014}.
\end{remark}

\subsection{Error analysis}
Now we carry out the error analysis of the fully discrete scheme \eqref{eqn:fully}, following the strategy outlined in the
pioneering work \cite{LubichSloanThomee:1996}. To derive $L^2\II$-error estimates, we split the error into
\begin{equation*}
  e^n=u(t_n)-U_h^n=(u(t_n)-u_h(t_n))+(u_h(t_n)-U_h^n).
\end{equation*}
In view of Theorems \ref{thm:nonsmooth-initial-op} and \ref{thm:smooth-initial-op}, it suffices to establish
a bound on $\|  u_h(t_n)-U_h^n \|_{L^2\II}$. The proof relies on the following splitting
\begin{equation*}
 u_h(t_n)-U_h^n= y_h(t)-Y_h^n,
\end{equation*}
where
\begin{equation*}
   y_h(t)=u_h(t)-v_h \quad \text{and} \quad Y_h^n=U_h^n-v_h.
\end{equation*}

First, we derive representations of the semidiscrete solution $y_h$ and fully discrete solution $Y_h$.
\begin{lemma}\label{lem:solurep}
Let the kernel $K(z)$  be defined by
\begin{equation}\label{eqn:Kernel}
K(z)=-z^{-1}(zw(z) I + A_h)^{-1}A_h
\end{equation}
and $\chi(z)=\frac{1-e^{-z\tau}}{\tau}$. Then $y_h$ and $Y_h^n$ can be represented by
\begin{equation*}
  y_h(t)=\frac{1}{2\pi \mathrm{i}}\int_{\Gamma_{\theta,\delta}} e^{zt}K(z)v_h dz\quad \mbox{and}\quad
  Y_h^n= \frac{1}{2\pi\mathrm{i} }\int_{\Gamma_\tau}e^{zt_{n-1}}K(\chi(z))v_h\,dz,
\end{equation*}
respectively, with the contour $\Gamma_\tau=\{ z\in \Gamma_{\theta,\delta}:|\Im(z)|\le {\pi}/{\tau} \}$.
\end{lemma}
\begin{proof}
By its definition, $y_h$ satisfies the problem:
\begin{equation*}
    \Dmu y_h  + A_hy_h = -A_h v_h,
\end{equation*}
with $y_h(0)=0$. The Laplace transform gives
\begin{equation*}
    zw(z) \hy_h(z)+ A_h\hy_h(z) = -z^{-1} A_hv_h.
\end{equation*}
Hence, $\hy_h(z)= K(z) v_h$, with $K(z)=-z^{-1}(zw(z) I + A_h)^{-1}A_h$, and the desired representation for $y_h(t)$ follows from
the inverse Laplace transform. Next, the fully discrete solution $Y_h^n$ satisfies the following time stepping scheme
\begin{equation*}
    Q_n(Y_h)+ A Y_h^n=-A_hv_h,
\end{equation*}
with $Y_h^0=0$. Now multiplying both sides by $\xi^n$, summing from $1$ to $\infty$ and noting $Y_h^0=0$ yield
\begin{equation*}
  \sum_{n=1}^\infty Q_n(Y_h)\xi^n + A_h \widetilde Y_h(\xi) = - %\widetilde 1_{s}
              \xi/(1-\xi) A_hv_h.
\end{equation*}
Using the condition $Y_h^0=0$, we have
\begin{equation*}
   \sum_{n=1}^\infty Q_n(Y_h) \xi^n
   = \sum_{n=0}^\infty \sum_{j=0}^n \left(b_{n-j}\xi^{n-j}\right)\left(Y_h^j\xi^j\right)
   = ((1-\xi)/\tau)w((1-\xi)/\tau) \widetilde Y_h(\xi).
\end{equation*}
Thus, by simple calculation, we deduce
\begin{equation*}
   \widetilde Y_h(\xi)= (\xi/\tau) K((1-\xi)/\tau)v_h,
\end{equation*}
and it is analytic at $\xi=0$. Then Cauchy theorem implies that for $\varrho$ small enough, there holds
\begin{equation*}
    Y_h^n = \frac{1}{2\tau \pi\mathrm{i}}\int_{|\xi|=\varrho} \xi^{-n} K((1-\xi)/\tau)v_h \, d\xi.
\end{equation*}
Now, by changing variable $\xi=e^{-z\tau},$ we obtain
\begin{equation*}
    Y_h^n = \frac{1}{2\pi\mathrm{i}}\int_{\Gamma_0} e^{zt_{n-1}} K((1-e^{-z\tau})/\tau)v_h \, dz,
\end{equation*}
where the contour $\Gamma_0=\{ z=-\ln(\varrho)/\tau+\mathrm{i} y:|y|\le {\pi}/{\tau} \}$ is oriented
counterclockwise. We obtain the desired representation by deforming the contour
$ \Gamma_0$ to $ \Gamma_\tau=\{ z\in \Gamma_{\theta,\delta}:|\Im(z)|\le {\pi}/{\tau} \}$
and using the periodicity of the exponential function.
\end{proof}

By Lemma \ref{lem:solurep}, we can write the difference between $Y_h^n$ and $y_h(t_n)$ as
$$ y_h(t_n)-Y_h^n= I+II,$$
where the terms $I$ and $II$ are given by
\begin{equation}\label{eqn:split1}
  I=\frac{1}{2\pi\mathrm{i}}\int_{\Gamma_{\theta,\delta}\backslash\Gamma_\tau} e^{zt_n} K(z) v_hdz
\end{equation}
and
\begin{equation}\label{eqn:split2}
  II=\frac{1}{2\pi\mathrm{i}}\int_{\Gamma_\tau} e^{zt_n} \left(K(z)-e^{-z\tau} K(\chi(z))\right)v_hdz.
\end{equation}
This splitting is essential for the error analysis below.
Since the function $|e^{-z\tau}|$ is uniformly bounded on the contour $\Gamma_{\tau}$, we have
\begin{equation}\label{split1}
\begin{split}
    \|  K(z)-e^{-z\tau} K(\chi(z)) \| & \le |e^{-z\tau}|\|  K(z)-K(\chi(z)) \| + |1-e^{-z\tau} |\|K(z)\| \\
    & \le c\|  K(z)-K(\chi(z)) \| + c\tau  |z|\|K(z)\|\\
    & \le c\|K(z)-K(\chi(z))\| + c\tau,
\end{split}
\end{equation}
where the last line, using the resolvent estimate \eqref{eqn:resol}, follows from the inequality
\begin{equation*}
  \|K(z)\| = |z|^{-1}\|-I+zw(z)(zw(z)+A_h)^{-1}\|\leq c|z|^{-1}.
\end{equation*}

Hence, it remains to bound the term $\|K(z)-K(\chi(z))\|$, which will be carried out in several steps.
First we recall a bound on the function $\chi(z)=\tau^{-1}(1-e^{-z\tau})$ \cite[Lemma 3.1]{JinLazarovZhou:2014L1}.
\begin{lemma}\label{lem:para_equiv1}
Let $\chi(z)=\tau^{-1}(1-e^{-z\tau})$. Then for all $z\in\Gamma_\tau$, there hold
\begin{equation*}
 |\chi(z)-z|\le c |z|^2\tau \quad \mbox{and}\quad c_1|z| \le |\chi(z)|\le c_2|z|,
\end{equation*}
and $\chi(z)$ lies in a sector $\Sigma_{\theta'}$ for some $\theta'\in(\pi/2,\pi)$.
\end{lemma}

Next we give one crucial error estimate on the approximation $\chi(z)w(\chi(z))$ to the kernel $zw(z)$.
\begin{lemma}\label{lem:key1}
For $z\in \Gamma_\tau$, the following bound holds:
\begin{equation*}
|\chi(z)w(\chi(z))-zw(z)| \le c\tau|z|^2w(|z|).
\end{equation*}
\end{lemma}
\begin{proof}
By the intermediate value theorem, for $z\in \Gamma_\tau$, we have
\begin{equation*}
 | \chi(z)^\al-z^\al| =  \al \bigg|\int_z^{\chi(z)} s^{\al-1} \,ds\bigg| \le \al |\chi(z)-z| \max_{\eta\in[0,1]}|z_\eta|^{\al-1} ,
\end{equation*}
where $z_\eta=\eta\chi(z)+ (1-\eta) z$ with $\eta\in [0,1]$. Next we claim $|z_\eta|^{-1}\le c|z|^{-1}$
for $\eta\in[0,1]$. To this end, we split $\Gamma_\tau$ into $\G_\tau=\G_\tau^+ \cup \G_{\tau}^c
\cup \G_\tau^-$, with $\G_\tau^\pm$ being the rays in the upper and lower half plane, respectively, and
$\Gamma_\tau^c$ is the circular arc.
For $z\in \Gamma_\tau^c$, by the Taylor expansion of $e^{-z\tau}$, we have
\begin{equation*}
 z_\eta = z \left(1 + \eta \sum_{j=1}^\infty (-1)^{j}\frac{z^{j}\tau^j}{(j+1)!}\right).
\end{equation*}
In view of the trivial inequality $|z\tau|\leq 1$ for $z\in\G_\tau^c$, we deduce
$| z_\eta |^{-1}\leq c|z|^{-1}$ for $ z\in \Gamma^c_\tau$. It remains to show the assertion for
$ z\in \G_\tau^+ $, and the case $z\in \G_\tau^-$ follows analogously. First we show
$\Im(\chi(z)) > 0 $ for $ z\in \G_\tau^+ $. For $z=re^{\mathrm{i}(\pi-\theta)}$ with $r\tau\in
(\delta,\pi/\sin\theta)$ we have
\begin{equation*}
    \chi(z)=\frac{1}{\tau}\left(1-e^{r\tau\cos\theta}e^{-\mathrm{i}r\tau\sin\theta}\right),
\end{equation*}
and therefore using $r\tau\sin\theta\le \pi$, we get $ \Im(\chi(z)) \geq 0 $. Then Lemma \ref{lem:para_equiv1} yields
\begin{equation*}
    |z_\eta|>\min(|z|,|\chi(z)|){\cos\frac{\theta}{2}}\ge c|z|.
\end{equation*}
This shows the desired claim. Hence, appealing to Lemma \ref{lem:para_equiv1} again implies that for $z\in \Gamma_\tau$ there holds
\begin{equation*}
\begin{split}
  \bigg|\int_0^1 (\chi(z)^\al-z^\al)\mu(\al)\,d\al\bigg| &\le \int_0^1 | \chi(z)^\al-z^\al| \mu(\al)\,d\al \le c\tau|z|\int_0^1 |z|^{\al}\mu(\al) \,d\al=c\tau|z|^2w(|z|),
\end{split}
\end{equation*}
which concludes the proof of the lemma.
\end{proof}

Next we give a crucial error estimate on the approximation $K(\chi(z))$ to the kernel function $K(z)$.
\begin{lemma}\label{lem:K}
Let $\chi(z)=(1-e^{-z\tau})/\tau$. Then for the kernel $K(z)$ in \eqref{eqn:Kernel}, there holds
\begin{equation*}
 \|K(z) - K(\chi(z))\| \le c\tau\quad \forall z\in \Gamma_\tau.
\end{equation*}
\end{lemma}
\begin{proof}
Let $B(z)=zK(z)$. Simple computation shows
\begin{equation*}
    \begin{split}
      B(z) - B(\chi(z)) &= zw(z)\left(zw(z)I+A_h\right)^{-1}-\chi(z) w(\chi(z)) \left(\chi(z) w(\chi(z))I+A_h\right)^{-1} \\
        &= zw(z)\left( \left(zw(z)I+A_h\right)^{-1} - \left(\chi(z) w(\chi(z))I+A_h\right)^{-1} \right)  \\
        &\quad + \left(zw(z)-\chi(z)w(\chi(z))\right)\,\left(\chi(z) w(\chi(z))I+A_h\right)^{-1} := I + II.
    \end{split}
\end{equation*}
First, by Lemmas \ref{lem:key2} and \ref{lem:para_equiv1}, there holds
\begin{equation*}
  |\chi(z)w(\chi(z))| \geq c|\chi(z)|w(|\chi(z)|) \geq c|z|w(|z|).
\end{equation*}
Further, by Lemma \ref{lem:sector-w} and \eqref{eqn:resol} and Lemma \ref{lem:key2}, we have
\begin{equation}\label{eqn:resol-wz}
  \|(zw(z)I+A_h)^{-1}\| \leq c|zw(z)|^{-1}\leq c(|z|w(|z|))^{-1}.
\end{equation}
Likewise, in view of Lemmas \ref{lem:para_equiv1} and \ref{lem:sector-w} and \eqref{eqn:resol}, we have
\begin{equation}\label{eqn:resol-chi}
  \|(\chi(z)w(\chi(z))I+A_h)^{-1}\| \leq c|\chi(z)w(\chi(z))|^{-1}\leq c(|z|w(|z|))^{-1}.
\end{equation}
Now, by the identity
\begin{equation*}
  \begin{aligned}
    &\left(zw(z)I+A_h\right)^{-1} - \left(\chi(z) w(\chi(z))I+A_h\right)^{-1} \\
    = &\left(zw(z)-\chi(z)w(\chi(z))\right)\,\left(zw(z)I+A_h\right)^{-1}\left(\chi(z)w(\chi(z))I+A_h\right)^{-1},
  \end{aligned}
\end{equation*}
Lemma \ref{lem:key2}, \eqref{eqn:resol-wz} and \eqref{eqn:resol-chi}, the first term $I$ can be bounded by
\begin{equation*}
    \| I \| \le c\tau |z|^3 w(|z|)^2 \|(zw(z)I+A_h)^{-1}\|\|(\chi(z)w(\chi(z))I+A_h)^{-1}\|\le c\tau |z|.
\end{equation*}
Likewise, with Lemma \ref{lem:key1} and \eqref{eqn:resol-chi}, the second term $II$ can be bounded by
\begin{equation*}
    \begin{split}
      \| II \| & \le   |zw(z)-\chi(z)w(\chi(z))| \| (\chi(z)w(\chi(z))I+A_h)^{-1}\|   \\
     & \le c\tau |z|^2 w(|z|) |zw(|z|)|^{-1}  \le c\tau |z|.
    \end{split}
\end{equation*}
Hence we bound $\| B(z) - B(\chi(z)) \|$ by
\begin{equation*}
 \| B(z) - B(\chi(z)) \| \le c\tau|z|.
\end{equation*}
Last, by Lemma \ref{lem:para_equiv1} and $\| B(z) \| \le c$, we bound $\| K(z) - K(\chi(z)) \|$ by
\begin{equation*}
    \begin{split}
      \|  K(z) - K(\chi(z)) \| & \le   |z^{-1}-\chi(z)^{-1}| \| B(z)\| + |z|^{-1} \| B(z) - B(\chi(z)) \| \\
     & \le c |z-\chi(z)| |z|^{-2} + c\tau \le c\tau,
    \end{split}
\end{equation*}
which completes the proof of the lemma.
\end{proof}

Now we can state an error estimate on the time discretization error for nonsmooth initial data,
{\it i.e.}, $v\in L^2(\Om)$.
\begin{theorem}\label{thm:error-nonsmooth}
Let $u_h$ and $U_h^n$ be the solutions of problems \eqref{eqn:fdesemidis} and \eqref{eqn:fully} with
$v\in L^2(\Om)$, $U_h^0= v_h = P_hv$ and $f\equiv0$, respectively. Then there holds
\begin{equation*}
   \| u_h(t_n)-U_h^n \|_{L^2(\Om)} \le c \tau t_n^{-1}  \| v\|_{L^2\II}.
\end{equation*}
\end{theorem}
\begin{proof}
It suffices to bound the terms $I$ and $II$ defined in \eqref{eqn:split1} and \eqref{eqn:split2}, respectively.
We choose $\delta=t_n^{-1}$ in the contour $\Gamma_{\delta,\theta}$. By \eqref{eqn:resol} and direct calculation, we
bound the first term $I$ by
\begin{equation}\label{eqn:errI}
\begin{split}
    \| I \|_{L^2\II} &\le c \int_{\pi/(\tau\sin\theta)}^\infty  e^{rt_n\cos\theta } r^{-1} \,dr \| v_h\|_{L^2\II}\\
    & \le c\tau \|v_h\|_{L^2\II} \int_{0}^\infty  e^{rt_n\cos\theta} \,dr \le c\tau t_n^{-1} \| v_h\|_{L^2\II}.
\end{split}
\end{equation}
Using Lemma \ref{lem:K}, we arrive at the following bound for the second term $II:$
\begin{equation}\label{eqn:errII}
\begin{split}
    \| II \|_{L^2\II}
    &\le c \tau \| v_h\|_{L^2\II} \left(\int_{1/t_n}^{\pi/(\tau\sin\theta)}  e^{r t_n\cos\theta}   \,dr +
    \int_{-\theta}^{\theta} e^{\cos\psi} t_n^{-1}\,d\psi\right)  \le c t_n^{-1}\tau \|v_h\|_{L^2\II}.
\end{split}
\end{equation}
Combining estimates \eqref{eqn:errI} and \eqref{eqn:errII} yields
\begin{equation*}
   \| y_h(t_n)-Y_h^n \|_{L^2(\Om)} \le c \tau t_n^{-1}  \|v_h\|_{L^2\II},
\end{equation*}
and the desired result follows directly from the identity $U_h^n-u_h(t_n)=Y_h^n-y_h(t_n)$
and the stability of the projection $P_h$ in $L^2(\Om)$.
\end{proof}

\begin{remark}\label{rem:stab}
The $L^2(\Om)$ stability of the time stepping scheme \eqref{eqn:fully} follows directly from Theorem \ref{thm:error-nonsmooth}.
\end{remark}

Next we turn to smooth initial data, {\it i.e.}, $Av\in L^2\II$. To this end, we first
state an alternative estimate on the solution kernel $K(z)$.
\begin{lemma}\label{lem:errKs}
Let  $K^s(z)=-z^{-1}(zw(z)I+A_h)^{-1}$. Then for any $z\in \Gamma_\tau$, there holds
\begin{equation*}
\| K^s(z)-K^s(\chi(z)) \|  \le  c\tau\frac{\log|z|}{|z|-1}.
\end{equation*}
\end{lemma}
\begin{proof}
Let $B^s(z)=-(zw(z)I+A_h)^{-1}$. Then by the trivial inequality
\begin{equation*}
  B^s(z)-B^s(\chi(z)) = \chi(z)w(\chi(z))-zw(z))   \left(zw(z)I +A_h\right)^{-1} \left(\chi(z)w(\chi(z))I+A_h\right)^{-1}
\end{equation*}
Lemma \ref{lem:key1}, and \eqref{eqn:resol-wz} and \eqref{eqn:resol-chi}, we deduce immediately
\begin{equation*}
\| B^s(z)-B^s(\chi(z)) \|  \le  c \tau  |w(z)|^{-1} .
\end{equation*}
Appealing to \ref{eqn:resol-wz} again, we have $\| B^s(z) \| \le c|zw(z)|^{-1}$, and thus
\begin{equation*}
    \begin{split}
      \|  K^s(z) - K^s(\chi(z)) \| & \le   |z^{-1}-\chi(z)^{-1}| \| B^s(z)\| + |\chi(z)|^{-1} \| B^s(z) - B^s(\chi(z)) \|\\
      &\le c|z-\chi(z)| |z|^{-3}|w(z)|^{-1} + c \tau  |zw(z)|^{-1}
      \le c \tau |zw(z)|^{-1}.
    \end{split}
\end{equation*}
Then the desired result follows from Lemma \ref{lem:key2}.
\end{proof}

Now we can state an error estimate for smooth initial data $Av \in L^2\II$.
\begin{theorem}\label{thm:error-smooth}
Let $u_h$ and $U_h^n$ be the solutions of problems \eqref{eqn:fdesemidis} and \eqref{eqn:fully} with
$Av\in L^2(\Om)$, $U_h^0= v_h=R_hv$ and $f\equiv0$, respectively. Then for $\ell_2(t)=\log \left(\max(t^{-1},2)\right)$, there holds
\begin{equation*}
    \| u_h(t_n)-U_h^n \|_{L^2(\Om)} \le c \tau \ell_2(t) \| Av \|_{L^2\II}.
\end{equation*}
\end{theorem}
\begin{proof}
Let $K^s(z)=-z^{-1}\left(zw(z)I+A_h\right)^{-1}$.
Then we can rewrite the error as
\begin{equation}\label{eqn:Vsplit2}
\begin{split}
   y_h(t_n) - Y_h^n=&\frac{1}{2\pi\mathrm{i}}\int_{\Gamma_{\theta,\delta}\backslash\Gamma_\tau} e^{zt_n} K^s(z)A_hv_hdz\\
&+\frac{1}{2\pi\mathrm{i}}\int_{\Gamma_\tau} e^{zt_n} \left(K^s(z)-e^{-z\tau} K^s(\chi(z))\right) A_hv_hdz := I+II.
\end{split}
\end{equation}
By Lemma \ref{lem:errKs} we have for $z\in\G_\tau$
\begin{equation*}
  \| K^s(z)-e^{-z\tau} K^s(\chi(z)) \| \le c\tau\frac{\log|z|}{|z|-1}.
\end{equation*}
By setting $\delta=1/t_n$ and by the monotonicity of the function $f(x)=\frac{\log(x)}{1-x}$ on $\mathbb{R}^+$,
we derive the following bound for the term $II$
\begin{equation*}\label{eqn:errIsmooth}
\begin{split}
    \| II \|_{L^2\II} &\le c \tau \| A_hv_h \|_{L^2\II}\left(\int_{1/t_n}^{\pi/(\tau\sin\theta)}  e^{r t_n\cos\theta} \frac{\log r}{r-1}dr
     + \int_{-\theta}^{\theta}  e^{\cos\psi} \frac{\log(t_n^{-1})}{1-t_n}d\psi \right)\le c \frac{\log(t_n^{-1})}{1-t_n} \tau \|A_hv_h\|_{L^2\II}.
\end{split}
\end{equation*}
Now \eqref{eqn:resol} implies that for all $z\in \G_{\theta, \delta}$, $\| K^s(z) \|\le c|z|^{-1}|zw(z)|^{-1}$. Therefore,
using Lemma \ref{lem:key2}, we deduce
\begin{equation}\label{eqn:errIIsmooth}
\begin{split}
    \| I \|_{L^2\II} &\le c \| A_hv_h \|_{L^2\II} \int_{\pi/(\tau\sin\theta)}^\infty  e^{r t_n\cos\theta} r^{-2}|w(r)|^{-1} \,dr\\
    &\le c \tau \| A_hv_h \|_{L^2\II} \int_{1/t_n}^\infty  e^{r t_n\cos\theta}\frac{\log r}{r-1}\,dr \le  c \frac{\log(t_n^{-1})}{1-t_n} \tau \|A_hv_h\|_{L^2\II}.
\end{split}
\end{equation}
Finally, we observe that if $t_n^{-1}\ge2$, {\it i.e.} $t_n\le 1/2$, then $\frac{\log(t_n^{-1})}{1-t_n}\le 2\log(t_n^{-1}).$
Otherwise if $t_n^{-1}<2$, {\it i.e.} $t_n\ge1/2$, then by the monotonicity of the function $f(x)=\frac{\log(x)}{1-x}$ on
$\mathbb{R}^+$, we deduce $\frac{\log(t_n^{-1})}{1-t_n}= \frac{\log(t_n)}{t_n-1}\le 2\log(2).$ Then the desired result follows
from \eqref{eqn:errIsmooth}, \eqref{eqn:errIIsmooth} and the identities $U_h^n-u_h(t_n)=Y_h^n-y_h(t_n)$ and $A_hR_h=P_hA$.
\end{proof}

The next theorem gives error estimates for the fully discrete scheme \eqref{eqn:fully}, which follow
from Theorems \ref{thm:nonsmooth-initial-op}, \ref{thm:smooth-initial-op}, \ref{thm:error-nonsmooth}
and \ref{thm:error-smooth} and the triangle inequality.
\begin{theorem}\label{thm:error_fully_CQ}
Let $u$ and $U_h^n$ be the solutions of problems \eqref{eqn:fde} and \eqref{eqn:fully} with
$U_h^0= v_h$ and $f\equiv0$, respectively. Then for  $\ell_1(t)=\log(2T/t)^{-1}$ and $\ell_2(t)=\log \left(\max(t^{-1},2)\right)$
and $t_n=n\tau$, the following error estimates hold.
\begin{itemize}
  \item[(a)] If $Av\in L^2\II$ and $v_h=R_h v$, then for $n \ge 1$
  \begin{equation*}
   \| u(t_n)-U_h^n \|_{L^2(\Om)} \le c(\tau \ell_2(t_n) + h^2)  \| Av \|_{L^2\II}.
  \end{equation*}
  \item[(b)] If $v\in L^2(\Om)$ and $v_h=P_hv$, then for $n \ge 1$
  \begin{equation*}
   \| u(t_n)-U_h^n \|_{L^2(\Om)} \le c_T \left (\tau + h^2 \ell_1(t_n) \right )t_n^{-1}  \| v\|_{L^2\II}.
  \end{equation*}
\end{itemize}
\end{theorem}

\begin{remark}
For distributed order time fractional diffusion, the error estimate involves a log factor in time for
smooth initial data, which  is reminiscent of the asymptotic behavior of the solution at small time, cf. Theorem
\ref{thm:reg}. This factor is not present for the single term and multi-term time fractional diffusion
\cite{JinLazarovLiuZhou:2014,JinLazarovZhou:2013}.
\end{remark}

%%%%%%%%%%%%%%%%%%%%%%%%%%%%%%%%%%%%%%%%%%%
\section{Numerical experiments and discussions}\label{sec:numer}
Now we present numerical results to verify the convergence theory. To this end, we let the domain $\Omega$ to be the unit internal
$\Om=(0,1)$ and consider the following three examples with smooth, discontinuous, and singular initial data:
\begin{itemize}
  \item[(a)] $v(x)=\sin(2\pi x) \in H^2(\Om)\cap H_0^1(\Om)$;
  \item[(b)] $v=\chi_{(0,1/2)}\in H^{{1/2}-\epsilon}(\Om)$ with $\epsilon\in(0,1/2)$,  and $\chi_{S}$
the characteristic function of a set $S$;
  \item[(c)] $v(x)=x^{-1/4}\in H^{{1/4}-\epsilon}(\Om)$ with $\epsilon\in(0,1/4)$.
\end{itemize}
We measure the temporal discretization error by the normalized $L^2\II$ errors $\| u(t_n)-U_{N,h}(t_n) \|_{L^2\II}/
\| v \|_{L^2\II}$ or $\| u(t_n)-U_h^n\|_{L^2\II}/\| v \|_{L^2\II}$, and the spatial discretization error by the
normalized $L^2\II$ and $H^1\II$ errors, \emph{i.e.}, $\|u(t)-u_h(t)\|_{L^2\II}/\|v\|_{L^2\II}$ and $\|\nabla(u(t)-u_h(t))
\|_{L^2\II}/\|v\|_{L^2\II}$. In the computations, we divide the domain $\Om$ into $M$ equally spaced subintervals with a mesh size
$h=1/M$. Since the exact solution $u(t)$ is not available in closed form, we compute the reference solution using a much
finer mesh.

\subsection{Numerical results for the semidiscrete scheme}

First we examine the convergence behavior of the space semidiscrete scheme. To this end, we fix $N=10$ in the Laplace transform approach
such that the error due to time discretization is negligible. The numerical results are given in Tables
\ref{tab:smooth-space}--\ref{tab:nonsmooth-space}. In the table, \texttt{rate} denotes the empirical convergence rates when the
mesh size $h$ halves, and the numbers in the bracket denote the theoretical rates. For all three initial data, the $L^2\II$ and $H^1\II$
norms of the error exhibit second and first order convergence rates, respectively, which agrees well with the theoretical prediction, \emph{cf.}
Theorems \ref{thm:nonsmooth-initial-op} and \ref{thm:smooth-initial-op}. The convergence of the semidiscrete scheme is robust in that
the convergence rates hold for both smooth and nonsmooth initial data. The error increases as $t\to0$, which is attributed to the weak
singularity of the solution as $t\to0$, \emph{cf.} Theorem \ref{thm:reg}.

\begin{table}[h!]
\caption{Numerical results for the standard semidiscrete Galerkin FEM for smooth initial data, Example (a)  with $N=10$ and $\mu(\al)=(\al-1/2)^2$.}
\label{tab:smooth-space}
\begin{center}
     \begin{tabular}{|c|c|cccccc|c|}
     \hline
   $t$&  %\backslashbox{norm\kern-50pt}
   {$M$} & 10 & 20 &40 &80 & 160 &320 & rate \\
     \hline
       & $L^2(\Om)$  &2.79e-5&7.02e-6&1.76e-6&4.39e-7&	1.09e-7&2.70e-8& 2.00 (2.00)\\
     \cline{2-8}
      \rb{1}         & $H^1(\Om)$  & 8.84e-4 & 4.44e-4 & 2.22e-4 & 1.11e-4 & 5.23e-5 & 2.36e-5 &1.05 (1.00)\\
     \hline
       & $L^2(\Om)$  & 2.40e-4 & 6.05e-5 & 1.52e-5 & 3.79e-6 & 9.35e-6& 2.33e-7&2.00 (2.00)\\
     \cline{2-8}
         \rb{$10^{-1}$}      & $H^1(\Om)$  & 7.03e-3 & 3.53e-3 & 1.77e-3 & 8.84e-4 & 4.16e-4 & 1.88e-4&1.00 (1.00)\\
     \hline
          & $L^2(\Om)$  & 6.38e-3 & 1.61e-3 & 4.03e-4 & 1.01e-4 & 2.51e-5& 6.21e-6&2.00 (2.00)\\
     \cline{2-8}
      \rb{$10^{-3}$}         & $H^1(\Om)$  & 1.41e-1 & 7.04e-2 & 3.53e-2 & 1.76e-2 & 3.75e-3 & 1.65e-3&1.07  (1.00)\\
     \hline
     \end{tabular}
\end{center}
\end{table}

\begin{table}[h!]
\caption{Numerical results for the standard semidiscrete Galerkin FEM for discontinuous initial data, Example (b) with $N=10$ and $\mu(\al)=(\al-1/2)^2$.}
\label{tab:discon-space}
\begin{center}
     \begin{tabular}{|c|c|cccccc|c|}
     \hline
     $t$&  %\backslashbox{norm\kern-30pt}
     {$M$} & 10 & 20 &40 &80 & 160 &320 & rate \\
     \hline
      & $L^2(\Om)$ & 3.97e-5&9.94e-6&2.48e-6&6.21e-7&	1.55e-7&3.87e-8& 2.00 (2.00)\\
     \cline{2-8}
     \rb{1}          & $H^1(\Om)$  & 1.26e-3 & 6.29e-4 & 3.15e-4 & 1.55e-4 & 7.63e-5 & 3.68e-5 & 1.01 (1.00)\\
     \hline
       & $L^2(\Om)$  & 5.81e-4 & 1.45e-4 & 3.64e-5 & 9.12e-6 & 2.28e-6& 5.69e-7& 2.00 (2.00)\\
     \cline{2-8}
        \rb{$10^{-2}$}        & $H^1(\Om)$  & 1.28e-2 & 6.38e-3 & 3.19e-3 & 1.57e-3 & 7.73e-4 & 3.73e-4& 1.02 (1.00)\\
     \hline
          & $L^2(\Om)$  & 6.34e-3 & 1.59e-3 & 3.96e-4 & 9.92e-5 & 2.48e-5& 6.18e-6& 2.00 (2.00)\\
     \cline{2-8}
       \rb{$10^{-3}$}         & $H^1(\Om)$  & 1.73e-1 & 8.65e-2 & 4.32e-2 & 2.14e-2 & 1.04e-2 & 5.06e-3&  1.02 (1.00)\\
     \hline
     \end{tabular}
\end{center}
\end{table}

\begin{table}[h!]
\caption{Numerical results for the standard semidiscrete Galerkin FEM for singular initial data, Example (c) with $N=10$ and $\mu(\al)=(\al-1/2)^2$.}
\label{tab:nonsmooth-space}
\begin{center}
     \begin{tabular}{|c|c|cccccc|c|}
     \hline
     $t$& %\backslashbox{norm\kern-30pt}
     {$M$} & 10 & 20 &40 &80 & 160 &320 & rate \\
     \hline
      & $L^2\II$ & 3.82e-5&9.67e-6&2.44e-6&6.12e-7&	1.53e-7&3.79e-8&  2.00 (2.00)\\
     \cline{2-8}
     \rb{1}          & $H^1(\Om)$  & 1.21e-3 & 6.13e-4 & 3.09e-4 & 1.55e-4 & 7.33e-5 & 3.33e-5 & 1.05 (1.00)\\
     \hline
       & $L^2(\Om)$  & 6.72e-4 & 1.69e-4 & 4.23e-5 & 1.06e-5 & 2.63e-6& 6.51e-7& 2.00 (2.00)\\
     \cline{2-8}
      \rb{$10^{-2}$}          & $H^1(\Om)$  & 1.38e-2 & 6.92e-3 & 3.47e-3 & 1.74e-3 & 8.18e-4 & 3.71e-4&  1.06 (1.00)\\
     \hline
          & $L^2(\Om)$  & 3.48e-3 & 8.76e-4 & 2.20e-4 & 5.49e-5 & 1.37e-5& 3.36e-6&  2.00 (2.00)\\
     \cline{2-8}
      \rb{$10^{-3}$}          & $H^1(\Om)$  & 1.49e-1 & 7.45e-2 & 3.73e-2 & 1.86e-2 & 8.76e-3& 3.97e-3 &  1.07 (1.00)\\
     \hline
     \end{tabular}
\end{center}
\end{table}

\subsection{Numerical results for the fully discrete scheme I}
Next we illustrate the convergence of the first fully discrete scheme based on the Laplace transform. To make the spatial discretization
error negligible, we fix the spatial mesh size $h$ at $h=10^{-5}$. In all numerical simulations, the optimal contour parameters
$\lambda$ and $\psi$ in the parameterization \eqref{eqn:curve} and $k$ in \eqref{eqn:trun-quad-symm} are chosen as suggested in
the proof of Lemma \ref{lem:quad-error} (see also \cite{WeidemanTrefethen:2007}). Moreover, $\lambda$ is fixed, independent of $t$,
with which the elliptic problems \eqref{eqn:ellip}
are solved for each time $t.$ The numerical results are summarized in Tables \ref{tab:time-error}
and \ref{tab:time-error-exp-chi} for the weight functions $\mu(\al)=(\al-1/2)^2$ and $\mu(\al)=\chi_{[1/2,1]}(\al)$, respectively.
The results indicate an exponential convergence with respect to the number $N$ of quadrature points on hyperbolic contour, decaying at
a rate about $e^{-2.15N}$ and $e^{-2.14N}$ for $\mu(\al)=(\al-1/2)^2$ and $\mu(\al)=\chi_{[1/2,1]}(\al)$, respectively, which agree
well with the theoretical predictions from Theorem \ref{thm:error_fully}. Note that even though the weight function $\mu(\alpha)=
\chi_{[1/2,1]}(\alpha)$ does not satisfy the assumption $\mu(0)\mu(1)>0$, the empirical convergence rates still agree well with the
theoretical prediction, which calls for further theoretical study. Further, the convergence rate is independent of time $t$, and thus
the scheme is robust. Interestingly, the smoothness of the initial data $v$ does not affect much the time discretization errors, even
for small time instances, \emph{cf.} Table \ref{tab:time-sing-smalltime-Laplace}.

\begin{table}[htb!]
\caption{The $L^2$ errors for initial data (a)-(c)  with $h=10^{-5}$ and $\mu(\al)=(\al-1/2)^2$, by the Laplace transform method. The
notation \texttt{$r$} denotes the exponential convergence rate in the error $\|u_{N,h}^n-u(t_n) \|_{L^2(\Om)} \le C e^{-rN}$.}
\label{tab:time-error}
\begin{center}
\vspace{-.3cm}{\setlength{\tabcolsep}{7pt}
     \begin{tabular}{|c|c|cccccc|c|}
     \hline
      case & $t\ \backslash\ N$ &$3$ &$5$ &$7$ & $9$ & $11$ & $13$ & $r$ \\
     \hline
      & 1  &1.33e-6 &1.49e-8 &1.26e-10 &2.20e-12 &3.54e-14 &8.24e-17 & 2.35\\
    (a)&  $10^{-2}$  &4.78e-6 &7.36e-7 &2.77e-9 &5.45e-11 &4.88e-13 &2.23e-14 & 1.92\\
               &  $10^{-3}$&8.30e-5 &8.78e-7 &3.81e-9 &7.55e-11 &6.43e-13 &1.23e-14 & 2.26\\
      \hline
      &  1  &3.34e-6 &3.56e-8 &2.85e-10 &5.76e-12 &8.68e-14 &1.25e-15 & 2.17\\
    (b)        &  $10^{-2}$  &1.24e-5 &8.29e-7 &2.31e-9 &6.09e-11 &4.78e-13 &2.18e-14 & 2.02\\
             &  $10^{-3}$   &6.99e-5 &1.73e-6 &1.09e-8 &5.38e-11 &1.17e-12 &1.59e-14 & 2.22\\
      \hline
      &   1   &8.04e-6 &9.05e-8 &6.80e-10 &1.39e-11 &2.08e-13 &3.02e-15 & 2.17\\
      (c)       & $10^{-2}$  &3.01e-5 &1.71e-6 &3.85e-9 &1.26e-10 &9.22e-13 &4.21e-14 & 2.04\\
             &  $10^{-3}$ &1.16e-4 &4.09e-6 &2.65e-8 &6.65e-11 &2.75e-12 &3.49e-14 & 2.19\\
      \hline
     \end{tabular}}
\end{center}
\end{table}

\begin{table}[htb!]
\caption{The $L^2$ errors for initial data (a)-(c)  with
  $h=10^{-5}$ and $\mu(\alpha)=\chi_{[1/2,1]}(\alpha)$, by the Laplace transform method.
The notation $r$ denotes the exponential convergence rate in the error $\|u_{N,h}^n-u(t_n) \|_{L^2(\Om)} \le C e^{-rN}$.}
\label{tab:time-error-exp-chi}
\begin{center}
\vspace{-.3cm}{\setlength{\tabcolsep}{7pt}
     \begin{tabular}{|c|c|cccccc|c|}
     \hline
      case & $t\ \backslash \ N$ &$3$ &$5$ &$7$ & $9$ & $11$ & $13$ & $r$ \\
     \hline
      & 1 &4.54e-6 &2.30e-7 &1.63e-9 &1.69e-11 &2.36e-13 &8.46e-15 & 2.02\\
    (a)&  $10^{-2}$  &6.21e-5 &1.65e-6 &3.71e-9 &1.07e-10 &7.00e-13 &2.58e-14 & 2.16\\
               &  $10^{-3}$&8.02e-4 &3.61e-6 &1.66e-8 &4.17e-10 &3.10e-12 &6.73e-15 & 2.55\\
      \hline
      &  1  &4.78e-6 &4.74e-7 &2.43e-9 &3.44e-11 &3.49e-13 &1.87e-14 & 1.94\\
    (b)        &  $10^{-2}$  &1.03e-4 &1.13e-6 &3.58e-9 &8.78e-11 &5.04e-13 &1.93e-14 & 2.24\\
             &  $10^{-3}$   &5.12e-4 &4.79e-6 &4.95e-8 &5.23e-10 &5.15e-12 &5.58e-14 & 2.29\\
      \hline
      &   1    &4.79e-6 &5.61e-7 &2.75e-9 &4.07e-11 &3.94e-13 &2.23e-14 & 1.92\\
     (c)       & $10^{-2}$  &1.18e-4 &6.08e-7 &3.37e-9 &7.22e-11 &2.84e-13 &8.94e-14 & 2.10\\
              &  $10^{-3}$ &1.09e-4 &5.24e-6 &6.02e-8 &5.62e-10 &5.95e-12 &1.02e-13 & 2.07\\
      \hline
     \end{tabular}}
\end{center}
\end{table}

\begin{table}[htb!]
\caption{The $L^2$ errors for initial data (b) and (c) with
 $h=10^{-5}$, $\mu(\al)=(\al-1/2)^2$ and $N=5$ at small time instances $t=10^{-k}$, $k =4,5,\cdots,9$, by the Laplace transform method.}
\label{tab:time-sing-smalltime-Laplace}
\begin{center}
\vspace{-.3cm}{\setlength{\tabcolsep}{7pt}
     \begin{tabular}{|c|cccccc|}
     \hline
     case $\backslash\ t$  & $10^{-4}$  & $10^{-5}$  & $10^{-6}$ & $10^{-7}$ & $10^{-8}$ & $10^{-9}$  \\
     \hline
      (b) &7.05e-6 &9.39e-6 &1.58e-5 &1.75e-5 &1.81e-5 &1.82e-5 \\
      \hline
      (c) &6.39e-6 &1.17e-5 &1.53e-5 &1.68e-5 &1.75e-5 &1.79e-5  \\
      \hline
     \end{tabular}}
\end{center}
\end{table}

One salient feature of the fully discrete scheme I is that it allows computing the solution at any arbitrarily
large time directly. This allows one to examine the asymptotic behavior of the solution as the time
$t\to\infty$; see Table \ref{tab:time-sing-largetime} and Fig. \ref{fig:soldecay-large}. In particular, one clearly
observes the logarithmic decay of the solution, as predicted by \cite[Theorem 2.1]{LiLuchkoYamamoto:2014}; see also Fig.
\ref{fig:soldecay-large}. This numerically verifies the ultraslow decay asymptotics for distributed order diffusion process,
in comparison with sublinear decay for subdiffusion and exponential decay for normal diffusion.

\begin{table}[htb!]
\caption{The $L^2$ norm of the solution for initial data (a) and (c) with
 $h=10^{-5}$, $\mu(\al)=(\al-1/2)^2$ and $N=10$ at large time instances $t=10^k$, $k =6,8,\cdots,18$, computed by the Laplace transform method.}
\label{tab:time-sing-largetime}
\begin{center}
\vspace{-.3cm}{\setlength{\tabcolsep}{7pt}
     \begin{tabular}{|c|ccccccc|c|}
     \hline
      case $\backslash \ k$  & 6 & 8 & 10 & 12 & 14 & 16& 18 &rate \\
     \hline
      (a) &3.33e-4 &2.70e-4 &2.26e-4 &1.95e-4 &1.71e-4 &1.52e-4 &1.37e-4 &$ 1/k$ \\
      \hline
      (c) &1.06e-3 &8.54e-4 &7.17e-4 &6.17e-4 &5.41e-4 &4.82e-4 &4.34e-4 &$1/k$ \\
      \hline
     \end{tabular}}
\end{center}
\end{table}

\begin{figure}[htb!]
  \centering
  \includegraphics[trim = .1cm .1cm .1cm 0.0cm, clip=true, width=10cm]{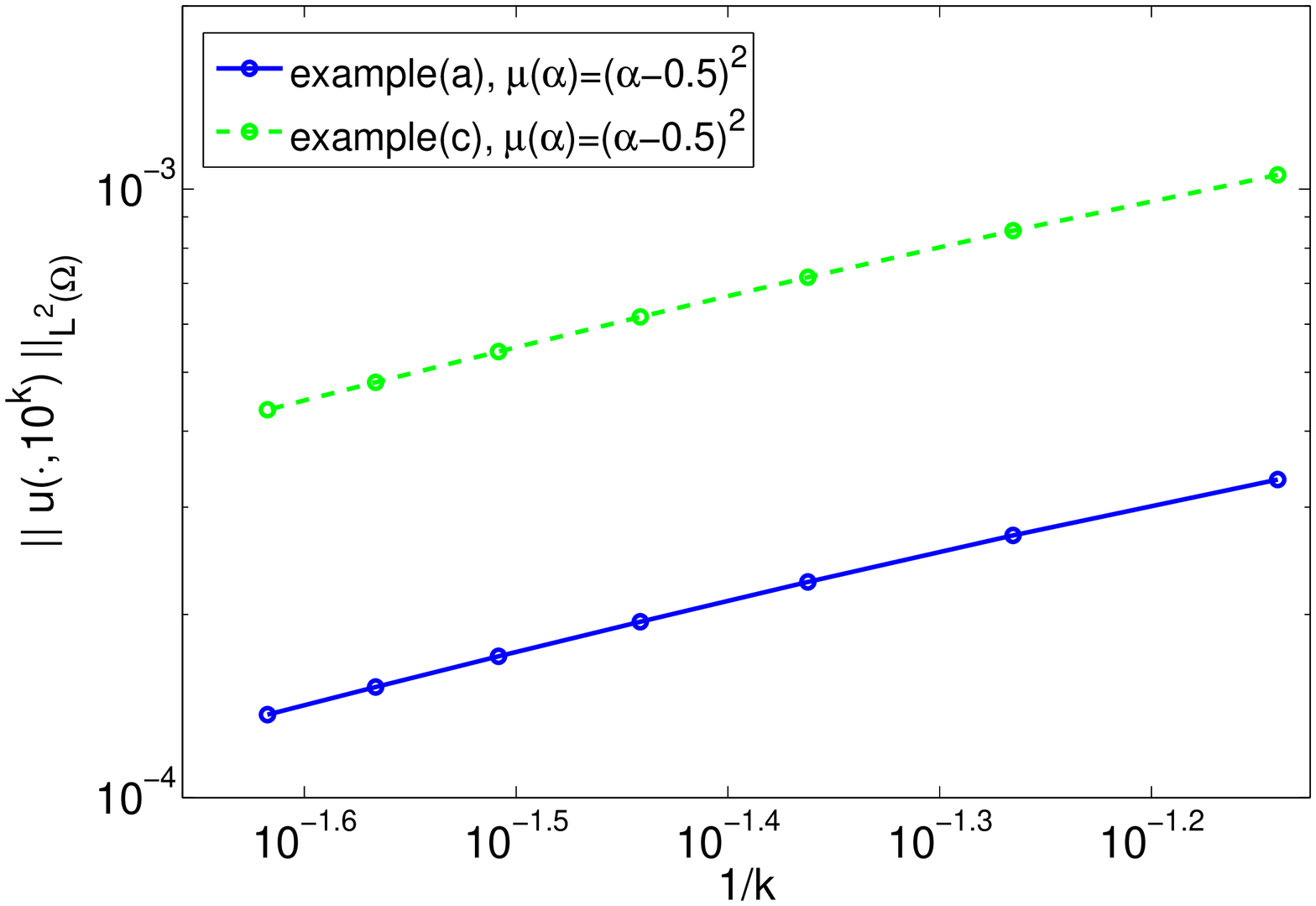}
  \caption{The $L^2$ norm of the solution for initial data (a) and (c) at $t=10^k$, $k = 6,8,\cdots,18$, by the Laplace transform method}.
  \label{fig:soldecay-large}
\end{figure}

\subsection{Numerical results for the fully discrete scheme II}

Last we verify the convergence of the fully discrete scheme II, {\it i.e.}, convolution quadrature based on the backward Euler
method. By Theorem \ref{thm:error_fully_CQ}, it exhibits a first order convergence with respect to the time step size $\tau$.
This is fully confirmed by the numerical results in Tables \ref{tab:time-error-cq} and \ref{tab:time-error-dchi} for the weight
functions $\mu(\al)=(\al-1/2)^2$ and $\mu(\alpha)=\chi_{[1/2,1]}(\alpha)$, respectively. A first order convergence is observed for all three
examples and at all time instances, showing the robustness of the scheme.

\begin{table}[htb!]
\caption{The $L^2$ errors for initial data (a)-(c)  with $h=10^{-4}$
 and $\mu(\al)=(\al-1/2)^2$, by the backward Euler convolution quadrature.}
\label{tab:time-error-cq}
\begin{center}
\vspace{-.3cm}{\setlength{\tabcolsep}{7pt}
     \begin{tabular}{|c|c|cccccc|c|}
     \hline
     case &  { $t \ \backslash\ N$} &10 & 20 &40 &80 & 160 &320 & rate \\
     \hline
      & 1 &1.82e-5 &8.78e-6 &4.31e-6 &2.12e-6 &1.01e-6 &4.74e-7 & 1.05 (1.00)\\
    (a)&  $10^{-2}$  &8.64e-4 &3.91e-4 &1.88e-4 &9.20e-5 &4.55e-5 &2.26e-5 &  1.05 (1.00)\\
               &  $10^{-3}$ &2.17e-2 &1.10e-2 &5.51e-3 &2.76e-3 &1.38e-3 &6.92e-4 &  0.99 (1.00)\\
      \hline
      &  1  &4.81e-5 &2.32e-5 &1.14e-5 &5.60e-6 &2.67e-6 &1.26e-6 &  1.05 (1.00)\\
    (b)        &  $10^{-2}$  &8.11e-3 &3.87e-3 &1.88e-3 &9.29e-4 &4.61e-4 &2.30e-4 &  1.03 (1.00)\\
             &  $10^{-3}$   &1.48e-2 &7.46e-3 &3.74e-3 &1.88e-3 &9.39e-4 &4.70e-4 &  1.00 (1.00)\\
      \hline
      &   1   &5.81e-5 &2.81e-5 &1.38e-5 &6.76e-6 &3.23e-6 &1.52e-6 &   1.05 (1.00)\\
      (c)       & $10^{-2}$  &1.01e-2 &4.80e-3 &2.34e-3 &1.15e-3 &5.72e-4 &2.85e-4 &  1.03 (1.00)\\
             &  $10^{-3}$ &7.35e-3 &3.66e-3 &1.82e-3 &9.11e-4 &4.55e-4 &2.27e-4 &  1.00 (1.00)\\
      \hline
     \end{tabular}}
\end{center}
\end{table}

\begin{table}[htb!]
\caption{The $L^2$ errors for initial data (a)-(c)  with $h=10^{-4}$ and $\mu(\alpha)=\chi_{[1/2,1]}(\alpha)$, by the backward Euler
convolution quadrature.}
\label{tab:time-error-dchi}
\begin{center}
\vspace{-.3cm}{\setlength{\tabcolsep}{7pt}
     \begin{tabular}{|c|c|cccccc|c|}
     \hline
     case & { $t\ \backslash\ N$} &10 & 20 &40 &80 & 160 &320 & rate \\
     \hline
      & 1 &2.20e-4 &1.06e-4 &5.20e-5 &2.58e-5 &1.28e-5 &6.40e-6 & 1.02 (1.00)\\
    (a)&  $10^{-2}$  &1.76e-2 &8.81e-3 &4.40e-3 &2.20e-3 &1.10e-3 &5.49e-4 &  1.00 (1.00)\\
               &  $10^{-3}$ &3.92e-3 &1.98e-3 &9.95e-4 &4.99e-4 &2.50e-4 &1.25e-4 &  0.99 (1.00)\\
      \hline
      &  1  &6.52e-4 &3.11e-4 &1.52e-4 &7.53e-5 &3.74e-5 &1.87e-5 &  1.03 (1.00)\\
    (b)        &  $10^{-2}$  &1.25e-2 &6.26e-3 &3.13e-3 &1.56e-3 &7.82e-4 &3.91e-4 &  1.00 (1.00)\\
             &  $10^{-3}$   &5.76e-3 &2.88e-3 &1.44e-3 &7.18e-4 &3.59e-4 &1.79e-4 &  1.00 (1.00)\\
      \hline
      &   1   &7.92e-4 &3.78e-4 &1.85e-4 &9.14e-5 &4.54e-5 &2.27e-5 &   1.03 (1.00)\\
      (c)       & $10^{-2}$  &7.40e-3 &3.71e-3 &1.86e-3 &9.28e-3 &4.64e-4 &2.32e-4 &  1.00 (1.00)\\
             &  $10^{-3}$ &6.10e-3 &3.06e-3 &1.53e-3 &7.65e-4 &3.83e-4 &1.91e-4 &  1.00 (1.00)\\
      \hline
     \end{tabular}}
\end{center}
\end{table}

To examine more closely the convergence behavior of the scheme, we consider $t=10^{-k}$, $k = 4,5,\cdots,9,$, and for each
time instance $t$, divide the interval $[0,t]$ into $N=10$ subintervals. The scheme works well for the smooth initial
data in example (a), however, it works poorly for the singular initial data  in example (c), \emph{cf.} Table \ref{tab:time-sing-smalltime}.
This behavior is predicted by Theorems \ref{thm:error-nonsmooth} and \ref{thm:error_fully_CQ}: the error is dominated by
the factor $\frac{\tau}{t}$ for $L^2(\Om)$ initial data. In Fig. \ref{fig:soldecay-small}, we plot the error ratio $\|U_h^1-u(\tau)\|/\tau$
against $\log \tau$ for smooth initial data in example (a). Theorem \ref{thm:error-smooth} predicts an error estimate
$\|U_h^1-u(\tau)\|_{L^2(\Om)}\leq c\tau \log \tau^{-1}$. The log factor $\ell_2(t)$ in Theorem \ref{thm:error-smooth}
is fully confirmed by Fig. \ref{fig:soldecay-small}, and thus the corresponding error estimate is sharp.

\begin{table}[htb!]
\caption{The $L^2$ errors for initial data (a) and (c) with
 $h=10^{-5}$ and $N=10$, at time instances $t=10^{-k}$, $k=4,5,\cdots,9$, by backward Euler convolution quadrature.}
\label{tab:time-sing-smalltime}
\begin{center}
\vspace{-.3cm}{\setlength{\tabcolsep}{7pt}
     \begin{tabular}{|c|cccccc|c|}
     \hline {case $\backslash\ t$} & $10^{-4}$  & $10^{-5}$  & $10^{-6}$ &
      $10^{-7}$ & $10^{-8}$ & $10^{-9}$ & rate \\
%      $t$  & 1e-4  & 1e-5  & 1e-6 & 1e-7 & 1e-8 & 1e-9 & rate \\
     \hline
      (a) &2.42e-3 &1.03e-4 &7.87e-6 &7.59e-7 &7.58e-8 &7.44e-9 &  1.01 (1.00) \\
      \hline
      (c) &7.44e-3 &5.67e-3 &4.30e-3 &3.27e-3 &2.49e-3 &1.88e-3 &   0.12 (0.12) \\
      \hline
     \end{tabular}}
\end{center}
\end{table}

\begin{figure}[htb!]
  \centering
  \includegraphics[trim = 0cm .1cm 0cm 0.0cm, clip=true, width=10cm]{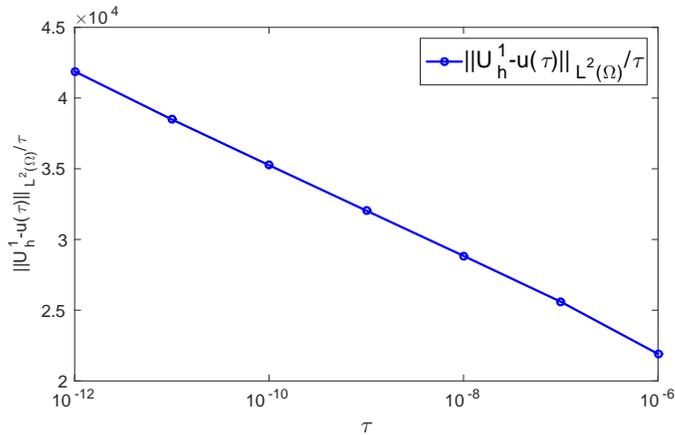}
  \caption{The $L^2$ errors for the backward Euler method for initial data (a) at small time instances $t_1=\tau=10^{-k}$, $k=5,6,...,11,12$. }
\label{fig:soldecay-small}
\end{figure}

\section{Concluding remarks}
In this work, we have presented a first rigorous  numerical analysis of two fully discrete schemes (one based on the Laplace transform and
another based on convolution quadratures) for the distributed-order time fractional diffusion equation
with nonsmooth initial data. We have provided regularity estimates for the solution and developed one space semidiscrete Galerkin
method and two fully discrete schemes. Optimal error estimates for the semidiscrete scheme were shown using an operator trick due
to Fujita and Suzuki. The first fully discrete scheme is based on quadrature approximation of the inverse Laplace transform with
a deformed contour of hyperbolic type, and exhibits an exponential convergence. It is especially suited to computing the solution
at many and large time instances. The second fully discrete scheme is based on convolution quadrature generated by the backward
Euler method, and exhibits a first order convergence. The sharpness of the error estimates were fully verified by extensive numerical
experiments for both smooth and nonsmooth initial data.

This work represents only a first step towards rigorous numerical analysis of distributed order subdiffusion, and there are a number
of avenues for further research. First, the semidiscrete and fully discrete schemes may be extended to the distributed order diffusion
wave equation, with a nonnegative weight $\mu(\alpha)\in C[0,2]$. Second, the error estimates for the semidiscrete Galerkin scheme
in the case of nonsmooth initial data $v\in L^2\II$ depend on the final time $T$. It remains unknown how to get rid of this factor.
This is especially important if the solution is sought for large $T$. Third, the assumption $\mu(\alpha)\in  C[0,1]$  might be too
restrictive and its is of much interest to relax it to $\mu(\alpha)\in L^\infty(0,1)$.

\section*{Acknowledgments}
The research of R. Lazarov and Z. Zhou have been supported in parts by NSF
Grant DMS-1016525 while that of D. Sheen by NRF--2014R1A2A1A11052429.

\bibliographystyle{abbrv}
\bibliography{frac}

\begin{thebibliography}{10}

\bibitem{AtanakovicPilipovicZorica:2011}
T.~M. Atanackovic, S.~Pilipovic, and D.~Zorica.
\newblock Distributed-order fractional wave equation on a finite domain.
  {S}tress relaxation in a rod.
\newblock {\em Internat. J. Engrg. Sci.}, 49(2):175--190, 2011.

\bibitem{Bazhlekova_2015arXiv}
E.~{Bazhlekova}.
\newblock Completely monotone functions and some classes of fractional
  evolution equations.
\newblock preprint, arXiv:1502.04647, 2015.

\bibitem{BazhlekovaJinLazarovZhou:2014}
E.~Bazhlekova, B.~Jin, R.~Lazarov, and Z.~Zhou.
\newblock An analysis of the {R}ayleigh-{S}tokes problem for the generalized
  second grade fluid.
\newblock {\em Numer. Math.}, 2014.
\newblock DOI-10.1007/s00211-014-0685-2 (arXiv:1401.8049).

\bibitem{Caputo:2001}
M.~Caputo.
\newblock Distributed order differential equations modelling dielectric
  induction and diffusion.
\newblock {\em Fract. Calc. Appl. Anal.}, 4(4):421--442, 2001.

\bibitem{ChechkinGorenfloSokolov:2002}
A.~V. Chechkin, R.~Gorenflo, and I.~M. Sokolov.
\newblock Retarding subdiffusion and accelerating superdiffusion governed by
  distributed-order fractional diffusion equations.
\newblock {\em Phys. Rev. E}, 66:046129, 2002.

\bibitem{ChechkinGorenfloSokolovGonchar:2003}
A.~V. Chechkin, R.~Gorenflo, I.~M. Sokolov, and V.~Y. Gonchar.
\newblock Distributed order time fractional diffusion equation.
\newblock {\em Fract. Calc. Appl. Anal.}, 6(3):259--279, 2003.

\bibitem{CuestaLubichPlencia:2006}
E.~Cuesta, C.~Lubich, and C.~Palencia.
\newblock Convolution quadrature time discretization of fractional
  diffusion-wave equations.
\newblock {\em Math. Comp.}, 75(254):673--696, 2006.

\bibitem{DiethelmFord:2009}
K.~Diethelm and N.~J. Ford.
\newblock Numerical analysis for distributed-order differential equations.
\newblock {\em J. Comput. Appl. Math.}, 225(1):96--104, 2009.

\bibitem{FujitaSuzuki:1991}
H.~Fujita and T.~Suzuki.
\newblock Evolution problems.
\newblock In {\em Handbook of {N}umerical {A}nalysis, {V}ol.\ {II}}, pages
  789--928. North-Holland, Amsterdam, 1991.

\bibitem{gavrilyuk2002mathcal}
I.~P. Gavrilyuk, W.~Hackbusch, and B.~N. Khoromskij.
\newblock $\mathcal{H}$-matrix approximation for the operator exponential with
  applications.
\newblock {\em Numer. Math.}, 92(1):83--111, 2002.

\bibitem{GorenfloLuchkoStajanovic:2013}
R.~Gorenflo, Y.~Luchko, and M.~Stojanovi{\'c}.
\newblock Fundamental solution of a distributed order time-fractional
  diffusion-wave equation as probability density.
\newblock {\em Fract. Calc. Appl. Anal.}, 16(2):297--316, 2013.

\bibitem{Hasse_book}
M.~Hasse.
\newblock {\em The {F}unctional {C}alculus for {S}ectorial {O}perators}.
\newblock Birkh\"auser-Verlag, Basel, 2006.

\bibitem{JiaPengLi:2014}
J.~Jia, J.~Peng, and K.~Li.
\newblock Well-posedness of abstract distributed-order fractional diffusion
  equations.
\newblock {\em Commun. Pure Appl. Anal.}, 13(2):605--621, 2014.

\bibitem{JinLazarovLiuZhou:2014}
B.~Jin, R.~Lazarov, Y.~Liu, and Z.~Zhou.
\newblock The {G}alerkin finite element method for a multi-term time-fractional
  diffusion equation.
\newblock {\em J. Comput. Phys.}, 281:825--843, 2015.

\bibitem{JinLazarovZhou:2013}
B.~Jin, R.~Lazarov, and Z.~Zhou.
\newblock Error estimates for a semidiscrete finite element method for
  fractional order parabolic equations.
\newblock {\em SIAM J. Numer. Anal.}, 51(1):445--466, 2013.

\bibitem{JinLazarovZhou:2014a}
B.~Jin, R.~Lazarov, and Z.~Zhou.
\newblock On two schemes for fractional diffusion and diffusion-wave equations.
\newblock preprint, arXiv:1404.3800, 2014.

\bibitem{JinLazarovZhou:2014L1}
B.~Jin, R.~Lazarov, and Z.~Zhou.
\newblock An analysis of the {L1} scheme for the subdiffusion equation with
  nonsmooth data.
\newblock {\em IMA Numer. Anal.}, page in press, 2015.

\bibitem{JinRundell:2014}
B.~Jin and W.~Rundell.
\newblock A tutorial on inverse problems for anomalous diffusion processes.
\newblock {\em Inverse Problems}, 31(3):035003, 40p., 2015.

\bibitem{Katsikadelis:2014}
J.~T. Katsikadelis.
\newblock Numerical solution of distributed order fractional differential
  equations.
\newblock {\em J. Comput. Phys.}, 259:11--22, 2014.

\bibitem{KilbasSrivastavaTrujillo:2006}
A.~Kilbas, H.~Srivastava, and J.~Trujillo.
\newblock {\em Theory and {A}pplications of {F}ractional {D}ifferential
  {E}quations}.
\newblock Elsevier, Amsterdam, 2006.

\bibitem{Kochubei:2008}
A.~N. Kochubei.
\newblock Distributed order calculus and equations of ultraslow diffusion.
\newblock {\em J. Math. Anal. Appl.}, 340(1):252--281, 2008.

\bibitem{LiLuchkoYamamoto:2014}
Z.~Li, Y.~Luchko, and M.~Yamamoto.
\newblock Asymptotic estimates of solutions to initial-boundary-value problems
  for distributed order time-fractional diffusion equations.
\newblock {\em Frac. Calc. Appl. Anal.}, 17(4):1114--1136, 2014.

\bibitem{LinXu:2007}
Y.~Lin and C.~Xu.
\newblock Finite difference/spectral approximations for the time-fractional
  diffusion equation.
\newblock {\em J. Comput. Phys.}, 225(2):1533--1552, 2007.

\bibitem{Lopez-FernandezPalencia:2006}
M.~{L{\'o}pez-Fern{\'a}ndez}, C.~Palencia, and A.~Sch{\"a}dle.
\newblock A spectral order method for inverting sectorial {L}aplace transforms.
\newblock {\em SIAM J. Numer. Anal.}, 44(3):1332--1350, 2006.

\bibitem{Lubich:1988}
C.~Lubich.
\newblock Convolution quadrature and discretized operational calculus. {I}.
\newblock {\em Numer. Math.}, 52(2):129--145, 1988.

\bibitem{LubichSloanThomee:1996}
C.~Lubich, I.~H. Sloan, and V.~Thom{\'e}e.
\newblock Nonsmooth data error estimates for approximations of an evolution
  equation with a positive-type memory term.
\newblock {\em Math. Comp.}, 65(213):1--17, 1996.

\bibitem{Luchko:2009}
Y.~Luchko.
\newblock Boundary value problems for the generalized time fractional diffusion
  equation of distributed order.
\newblock {\em Frac. Cal. Appl. Anal.}, 12(4):409--422, 2009.

\bibitem{mainardi2008time}
F.~Mainardi, A.~Mura, G.~Pagnini, and R.~Gorenflo.
\newblock Time-fractional diffusion of distributed order.
\newblock {\em J. Vibr. Control}, 14(9--10):1267--1290, 2008.

\bibitem{Martensen:1968}
E.~Martensen.
\newblock Zur numerischen {A}uswertung uneigenlicher {I}ntegrale.
\newblock {\em Z. Angew. Math. Mech.}, 48:T83--T85, 1968.

\bibitem{McleanMustapha:2014}
W.~McLean and K.~Mustapha.
\newblock Time-stepping error bounds for fractional diffusion problems with
  non-smooth initial data.
\newblock J. Comput. Phys., in press. arXiv:1405.2140, 2014.

\bibitem{McLeanSloanThomee:2006}
W.~McLean, I.~H. Sloan, and V.~Thom{\'e}e.
\newblock Time discretization via {L}aplace transformation of an
  integro-differential equation of parabolic type.
\newblock {\em Numer. Math.}, 102(3):497--522, 2006.

\bibitem{mclean2010numerical}
W.~McLean and V.~Thom{\'e}e.
\newblock Numerical solution via {Laplace} transforms of a fractional order
  evolution equation.
\newblock {\em J. Integral Equations Appl}, 22(1):57--94, 2010.

\bibitem{MeerschaertNaneVellaisamy:2011}
M.~M. Meerschaert, E.~Nane, and P.~Vellaisamy.
\newblock Distributed-order fractional diffusions on bounded domains.
\newblock {\em J. Math. Anal. Appl.}, 379(1):216--228, 2011.

\bibitem{MeerschaertScheffler:2006}
M.~M. Meerschaert and H.-P. Scheffler.
\newblock Stochastic model for ultraslow diffusion.
\newblock {\em Stochastic Process. Appl.}, 116(9):1215--1235, 2006.

\bibitem{MorgadoRebelo:2015}
M.~L. Morgado and M.~Rebelo.
\newblock Numerical approximation of distributed order reaction diffusion
  equations.
\newblock {\em J. Comput. Appl. Math.}, 275:216--227, 2015.

\bibitem{MustaphaAbdallahFurati:2014}
K.~Mustapha, B.~Abdallah, and K.~M. Furati.
\newblock A discontinuous {Petrov--Galerkin} method for time-fractional
  diffusion equations.
\newblock {\em SIAM J. Numer. Anal.}, 52(2):2512--2529, 2014.

\bibitem{Podlubny:1999}
I.~Podlubny.
\newblock {\em Fractional {D}ifferential {E}quations}.
\newblock Academic Press, San Diego, CA, 1999.

\bibitem{SheenSloanThomee:2000}
D.~Sheen, I.~Sloan, and V.~Thom{\'e}e.
\newblock A parallel method for time-discretization of parabolic problems based
  on contour integral representation and quadrature.
\newblock {\em Math. Comp.}, 69(229):177--195, 2000.

\bibitem{SheenSloanThomee:2003}
D.~Sheen, I.~Sloan, and V.~Thom{\'e}e.
\newblock A parallel method for time discretization of parabolic equations
  based on {L}aplace transformation and quadrature.
\newblock {\em IMA J. Numer. Anal.}, 23(2):269--299, 2003.

\bibitem{Sinai:1982}
Y.~G. Sina{\u{\i}}.
\newblock The limit behavior of a one-dimensional random walk in a random
  environment.
\newblock {\em Teor. Veroyatnost. i Primenen.}, 27(2):247--258, 1982.

\bibitem{SokolovChechkinKlafter:2004}
I.~M. Sokolov, A.~V. Chechkin, and J.~Klafter.
\newblock Distributed-order fractional kinetics.
\newblock {\em Acta Phys. Polon. B}, 35(4):1323--1341, 2004.

\bibitem{Thomee:2006}
V.~Thom{\'e}e.
\newblock {\em Galerkin {F}inite {E}lement {M}ethods for {P}arabolic
  {P}roblems}, volume~25 of {\em Springer Series in Computational Mathematics}.
\newblock Springer-Verlag, Berlin, 2006.

\bibitem{VergaraZacher:2015}
V.~Vergara and R.~Zacher.
\newblock Optimal decay estimates for time-fractional and other nonlocal
  subdiffusion equations via energy methods.
\newblock {\em SIAM J. Math. Anal.}, 47(1):210--239, 2015.

\bibitem{WeidemanTrefethen:2007}
J.~A.~C. Weideman and L.~N. Trefethen.
\newblock Parabolic and hyperbolic contours for computing the {B}romwich
  integral.
\newblock {\em Math. Comp.}, 76(259):1341--1356, 2007.

\bibitem{ZengLiLiuTurner:2013}
F.~Zeng, C.~Li, F.~Liu, and I.~Turner.
\newblock The use of finite difference/element approaches for solving the
  time-fractional subdiffusion equation.
\newblock {\em SIAM J. Sci. Comput.}, 35(6):A2976--A3000, 2013.

\bibitem{ZhangSunLiao:2014}
Y.-N. Zhang, Z.-Z. Sun, and H.-L. Liao.
\newblock Finite difference methods for the time fractional diffusion equation
  on non-uniform meshes.
\newblock {\em J. Comput. Phys.}, 265:195--210, 2014.

\end{thebibliography}
\end{document}